\documentclass[a4paper]{article}

\input{preamble}

\title{\bfseries Elastically safeguarded augmented Lagrangian methods}
\author{Ernesto G.\ Birgin\thanks{%
        Institute of Mathematics and Statistics, University of S\~ao Paulo, Brazil. \email{egbirgin@ime.usp.br}, \orcid{0000-0002-7466-7663}.}%
    \and Alberto De Marchi\thanks{%
        Institute of Applied Mathematics and Scientific Computing, Department of Aerospace Engineering, University of the Bundeswehr Munich, Germany. \email{alberto.demarchi@unibw.de}, \orcid{0000-0002-3545-6898}.}%
    \and Patrick Mehlitz\thanks{%
        Department of Mathematics and Computer Science, Marburg University, Germany. \email{mehlitz@uni-marburg.de}, \orcid{0000-0002-9355-850X}.}%
}
\newcommand{\TheFunding}{This work has been partially supported by the Brazilian agencies FAPESP (grants 2013/07375-0, 2022/05803-3, 2023/08706-1, and 2024/22384-0) and CNPq (grant 302073/2022-1).}
\date{}

\begin{document}

\maketitle

\begin{abstract}
	We investigate, theoretically and numerically, a class of elastically safeguarded augmented Lagrangian methods for nonlinear optimization problems with inequality and equality constraints.
	Safeguarded augmented Lagrangian methods are known to exhibit substantially stronger global convergence guarantees than their non-safeguarded counterparts, making them attractive in practice.
	A persistent limitation, however, is that existing methods rely on a fixed safeguard, whose a priori selection can be difficult and inherently limits adaptivity, e.g., with respect to problem scaling.
	We propose an adaptive safeguarding mechanism that allows the safeguard to grow dynamically, overcoming these drawbacks while preserving the desirable global convergence properties of variants with a fixed safeguard.
	We further establish complexity bounds in terms of worst-case iteration counts.
	Numerical experiments comparing ALGENCAN, a state-of-the-art solver with rigid safeguard, against an elastically safeguarded variant thereof confirm that elastic safeguarding yields consistent practical benefits.

    \medskip

    
    \noindent\textbf{AMS MSC.}
    \amsmscLink{49M20}, % Numerical methods of relaxation type
    \amsmscLink{65K05}, % Numerical mathematical programming methods
    \amsmscLink{65Y20}, % Complexity and performance of numerical algorithms
    \amsmscLink{90C30}. % Nonlinear programming
\end{abstract}

\blfootnote{\TheFunding}

\section{Introduction}

We are concerned with the standard nonlinear optimization problem
\begin{equation}\label{eq:NLP}\tag{NLP}
    \minimize_x\quad
    f(x)
    \quad\stt\quad
    g(x)\leq 0,\quad
    h(x)=0
\end{equation}
where $f\colon\R^n\to\R$, $g\colon\R^n\to\R^m$, and $h\colon\R^n\to\R^\ell$
are continuously differentiable functions.
Lower-level constraints of the form $x\in\Omega$ for some closed set $\Omega\subseteq \R^n$ can also be considered in \eqref{eq:NLP},
as in \cite{birgin2014practical,andreani2008augmented}.
Including such constraints would not affect the theoretical analysis, but it would make the notation more technical --- we prefer to drop them for simplicity and clarity of presentation.

\medskip

Augmented Lagrangian methods (ALMs) are among the most successful approaches for solving nonlinear programming problems of type \eqref{eq:NLP}.
Introduced independently by Hestenes \cite{hestenes1969multiplier} and Powell \cite{powell1969method}, and extended to convex programs by Rockafellar \cite{rockafellar1973multiplier,rockafellar1976augmented}, these methods
augment the Lagrangian with a quadratic penalty term and alternate between primal
minimization and multiplier updates, achieving exact stationarity for finite values of the
penalty parameter (under mild conditions).

The theoretical properties of ALMs depend critically on the multiplier update rule.
The classical Powell--Hestenes--Rockafellar (PHR) scheme uses an unconstrained update and typically holds the penalty parameter fixed, but its convergence theory requires either convexity of the model problem or boundedness of the multiplier sequence.
Without an additional safeguarding mechanism, global convergence typically fails \cite{kanzow2017example}.
The safeguarding scheme of Andreani, Birgin, Martínez, and Schuverdt \cite{andreani2008augmented,birgin2014practical} clips the multiplier estimates to a prescribed compact set at every iteration.
The mechanism of Conn, Gould, and Toint \cite{conn1991globally} updates the multiplier estimates only when some other conditions are satisfied. 
Both algorithms, implemented respectively in the software packages ALGENCAN and LANCELOT, are designed in such a way that the multiplier estimates do not ``behave too badly'' \cite[\S 4]{conn1991globally}, yielding global convergence for nonconvex problems.
However, these safeguards can slow down dual convergence and lead to strictly weaker guarantees than the classical PHR update.

\paragraph*{Contribution}
The present work resolves this tension.
We propose an ALM whose multiplier update is based on the elastic safeguarding scheme of \cite[\S 6]{demarchi2025augmented}, 
in which the safeguarding bound is adaptively relaxed so that it becomes inactive near regular limit points,
restoring PHR-type dual convergence precisely where it is most beneficial while retaining the robustness of safeguarding elsewhere.
We extend this mechanism from the convex composite setting of \cite{demarchi2025augmented} to nonconvex nonlinear programming, and contribute two further analytical results.
First, we provide a global convergence analysis showing that accumulation points are approximately stationary in the classical sense of Karush--Kuhn--Tucker (KKT) \cite[Def.~3.1]{birgin2014practical} and, under weak constraint qualifications, even stationary for \eqref{eq:NLP}.
Second, we establish worst-case complexity bounds on the number of (outer and total inner) iterations needed to produce either an $\epsilon$-approximate solution or a certificate of local infeasibility, 
in the spirit of Birgin and Martínez \cite{birgin2020complexity} for ALGENCAN.

\paragraph*{Outline}
The paper is organized as follows.
After surveying the relevant literature and introducing notation, \cref{sec:alm} describes the proposed algorithm.
\cref{sec:conv_ana,sec:complexity} present the global convergence and complexity results, respectively.
\cref{sec:num_experiments} reports some results of numerical experiments on the CUTEst collection \cite{Gould2014}.
Some final remarks conclude the paper in \cref{sec:conclusions}.

\subsection{Literature review}

\paragraph*{Origins and the connection to the proximal point algorithm}
The augmented Lagrangian method was introduced by Hestenes \cite{hestenes1969multiplier} and Powell \cite{powell1969method} for equality constraints and extended to inequality constraints and convex programs by Rockafellar \cite{rockafellar1973multiplier}, who also established the landmark equivalence between ALMs and the proximal point algorithm applied in dual space \cite{rockafellar1976augmented}.
This connection is central to the convergence theory of convex ALMs and carries over to inexact variants and beyond \cite{rockafellar1976monotone,rockafellar2022convergence,demarchi2025augmented}, see the monographs of Bertsekas \cite{bertsekas1996constrained} and Rockafellar and Wets \cite{rockafellar1998variational} for comprehensive treatments.

\paragraph*{Safeguarded ALMs for nonconvex problems}
For nonconvex problems, unbounded multiplier growth in the classical PHR scheme can cause dual divergence even when primal iterates converge to a feasible point.
Conn, Gould, and Toint \cite{conn1991globally} gave an early and influential resolution: multiplier estimates are updated only when the inner subproblem achieves sufficient reduction of infeasibility, and the penalty parameter is decreased otherwise.
This conditional update rule keeps the penalty parameter bounded away from zero under a strong regularity condition and was implemented in the LANCELOT package.
Andreani, Birgin, Martínez, and Schuverdt \cite{andreani2008augmented,birgin2014practical} later introduced a safeguarding mechanism based on imposing uniform bounds on the multiplier estimates, which underpins the widely used ALGENCAN solver.
Birgin, Castillo, and Martínez \cite{BirginCastilloMartinez2005} defined a general class of 65 safeguarded ALM variants and proved primal-dual global convergence under a strong regularity condition, establishing safeguarded ALMs as the standard for nonconvex programming.

\paragraph*{Global convergence under weak regularity}
Birgin, Haeser, Maculan, and Ramirez \cite{BirginHaeserMaculanRamirez2025} generalized \cite{BirginCastilloMartinez2005} to include non-safeguarded updates within a single convergence theory, covering the classical PHR method as a special case.
Their key contribution is a weak regularity condition that does not require the Lagrange multiplier set to be bounded.
Under it, every qualifying limit point is a stationary point, and non-safeguarded methods may keep the penalty parameter fixed throughout, even without convexity.
The present work builds on \cite{BirginHaeserMaculanRamirez2025}, showing that elastic safeguarding preserves these conclusions while sharpening them in the regular case.

\paragraph*{Complexity analysis}
Birgin and Martínez \cite{birgin2020complexity} established the first rigorous worst-case complexity bounds for ALGENCAN, 
giving $\bigO(\log^2(1/\epsilon))$ outer iterations to reach either an approximately stationary point or a certificate of local infeasibility.
Complexity bounds for the number of inner iterations, function evaluations, and derivative evaluations were also established under some Lipschitz continuity assumptions.
Our complexity analysis extends this framework to the elastic safeguarding setting. 
The adaptive safeguarding bound introduces additional case distinctions depending on whether the bound is active, but the logarithmic outer iteration count is preserved, with constants depending explicitly on the elasticity parameters.

\paragraph*{Elastic safeguarding}
The elastic scheme we employ originates in \cite[\S 6]{demarchi2025augmented}, which studies ALMs for fully convex composite problems.
That work extends the ALM--proximal point correspondence to the inexact safeguarded setting and distinguishes the regular case (a stationary minimizer exists) from the irregular case (all minimizers are nonstationary).
Standard safeguarding leads to weaker convergence guarantees than the ones for the classical ALM in the regular convex case,
where bounds for the multiplier estimates are superfluous, if not an impediment.
The elastic modification of \cite{demarchi2025augmented} remedies this by allowing the safeguarding bound to expand as the iterates approach a regular solution, recovering PHR behavior in the limit while preserving primal convergence guarantees in the irregular case.
We adapt this mechanism to the nonconvex programming setting, where the proximal point correspondence is unavailable, and establish global convergence results under a hierarchy of assumptions.

\paragraph*{Composite optimization}
The methodology and analysis investigated in this paper can be extended to problems of type
\[
\minimize_x\quad
f(x) + \varphi(F(x))
\quad\stt\quad
x\in\R^n
\]
where $f\colon\R^n\to\R$ and $F\colon\R^n\to\R^q$ are continuously differentiable
and $\varphi\colon\R^q \to \R\cup\{+\infty\}$ is lower semicontinuous and prox-friendly,
see \cite{demarchi2023constrained,demarchi2024local,demarchi2025augmented}
where safeguarded ALMs for such problems have been studied.
Obviously, problem \eqref{eq:NLP} is covered by this abstract setting
as $\varphi$ can be chosen to be the indicator function of a closed set.

\subsection{Notation}

The ceiling function takes as input a real number $\alpha\in\R$ and returns the least integer greater than or equal to $\alpha$, 
written $\ceil{\alpha}$.
The signum function takes $\alpha$ as input and returns $-1$, $0$, or $1$ if $\alpha<0$, $\alpha=0$, or $\alpha>0$, respectively,
written $\sgn(\alpha)$.
By $\N$, $\R$, $\R^n$, and $\R^n_+$,
we denote the natural numbers including $0$, the real numbers, the set of all vectors with $n$ real components,
and the nonnegative orthant in $\R^n$.
Throughout, $\R^n$ is equipped with the Euclidean norm $\norm{\cdot}$ and the Euclidean inner product $\innprod{\cdot}{\cdot}$.
The infinity norm is represented by $\norm{\cdot}_\infty$.
Given a set $A\subseteq\R^n$ and a scalar $\alpha\in\R$, their product is $\alpha A \coloneqq \{\alpha x \,|\, x\in A\}$.
For a nonempty, bounded set $B\subseteq\R^n$, 
we use $\extent(B)\coloneqq\sup_{x\in B}\norm{x}$ to denote the extent of $B$,
i.e., the radius of the smallest closed ball around the origin that contains $B$.
For $u,v\in\R^n$, vectors $\max\{u,v\},\min\{u,v\}\in\R^n$ are defined by applying the $\max$ and $\min$
operation componentwise, respectively, and $[u,v] \coloneqq \{x\in\R^n\,|\,u\leq x\leq v\}$
is the box induced by the lower bound $u$ and the upper bound $v$.
For a sequence $\{x^k\}\subseteq\R^n$ of vectors and $x\in\R^n$, $x^k\to x$ indicates convergence of $\{x^k\}$ to $x$.
Given an infinite set $K\subseteq\N$, $x^k\to_K x$ is used to express that the subsequence $\{x^k\}_{k\in K}$
converges to $x$.
For a real sequence $\{t_k\}$ and $t\in\R$, $t_k\to t$ indicates convergence of $\{t_k\}$ to $t$.
We use $t_k\searrow 0$ to express that $\{t_k\}$ converges to $t$ from above,
and $t_k\downarrow t$ additionally indicates that $t_k\neq t$ holds for all $k\in\N$.

Given a continuously differentiable vector function $\Psi\colon\R^n\to\R^m$ and $x\in\R^n$, 
$\Psi'(x)\in\R^{m\times n}$ is the Jacobian of $\Psi$ at $x$.
For a continuously differentiable scalar function $\psi\colon\R^n\to\R$,
$\nabla\psi(x) \coloneqq \psi'(x)^\top$ represents the gradient of $\psi$ at $x$.
Partial Jacobians and gradients with respect to certain blocks of variables
are denoted in the canonical way.

\section{Augmented Lagrangian framework}\label{sec:alm}

We review some foundations of augmented Lagrangian methods in \cref{sec:AL_foundations}
before presenting the elastically safeguarded variant in \cref{sec:elastic_safeguarding}.

\subsection{Foundations of augmented Lagrangian techniques}\label{sec:AL_foundations}

Throughout, let $\Lag\colon\R^n\times\R^m\times\R^\ell\to\R$ be the Lagrangian function of \eqref{eq:NLP}, defined by
\[
    \Lag(x,\mu,\lambda)
    \coloneqq
    f(x) + \innprod{\mu}{g(x)} + \innprod{\lambda}{h(x)},
\]
and recall that a feasible point $x\in\R^n$ of \eqref{eq:NLP} is referred to as stationary
if there exist Lagrange multipliers $\mu\in\R^m$
and $\lambda\in\R^\ell$ such that
\[
	\nabla_x\Lag(x,\mu,\lambda)=0,\quad
	\mu\geq 0,\quad
	\mu^\top g(x) = 0.
\] 
When additionally equipped with the feasibility conditions $g(x)\leq 0$ and $h(x)=0$,
the above system is referred to as the KKT conditions of \eqref{eq:NLP}.
It is well known that local minimizers of \eqref{eq:NLP} are stationary
in the presence of suitable qualification conditions like the
Mangasarian--Fromovitz constraint qualification (MFCQ), which claims
that
\[
	g'(x)^\top\mu + h'(x)^\top\lambda=0,\quad
	\mu\geq 0,\quad
	\mu^\top g(x) = 0
\]
merely possesses the trivial solution $(\mu,\lambda)=(0,0)$.

For some penalty parameter $\upsilon>0$,
we also make use of the so-called augmented Lagrangian (AL) function
$\AugLag_\upsilon\colon\R^n\times\R^m\times\R^\ell\to\R$ 
of \eqref{eq:NLP} which is defined by
\[
    \AugLag_\upsilon(x,\mu,\lambda)
    \coloneqq
    f(x)
    +
    \frac{1}{2\upsilon} \norm{ \max\{0, g(x)+\upsilon \mu\} }^2
    +
    \frac{1}{2\upsilon} \norm{ h(x)+\upsilon \lambda }^2.
\]
Given an iteration index $k\in\N$, the associated AL subproblem is
\begin{equation}\label{eq:subproblem}\tag{SP$_k$}
    \minimize_x \quad
    \AugLag_{\upsilon_k}(x,\hat{\mu}^k,\hat{\lambda}^k)
    \quad\stt\quad x\in\R^n
\end{equation}
for some given multiplier estimates $\hat{\mu}^k \in \R^m$, $\hat{\lambda}^k\in\R^\ell$ and penalty parameter $\upsilon_k>0$.
We assume throughout the existence of solutions for \eqref{eq:subproblem}, which can result from certain boundedness assumptions.
In practice, this can be guaranteed, e.g., by including in \eqref{eq:subproblem} simple constraints of the form $x\in[x_l,x_u]$ with $x_l,x_u\in\R^n$ such that $x_l\leq x_u$, see \cite[(1)--(2)]{birgin2020complexity}.
Given some tolerance $\varepsilon_k\geq 0$,
what can be expected is that an iterate $x^k$ obeys
the approximate global optimality
\begin{equation}\label{eq:approximate_minimality}
	\forall x\in\R^n\colon\quad
	\AugLag_{\upsilon_k}(x^k,\hat{\mu}^k,\hat{\lambda}^k)
    \leq
    \AugLag_{\upsilon_k}(x,\hat{\mu}^k,\hat{\lambda}^k) + \varepsilon_k
\end{equation}
for \eqref{eq:subproblem} whenever the data functions $f$ and $g_1,\ldots,g_m$ are convex and $h$ is affine.
In the absence of these additional properties of the data,
it is still reasonable to require the approximate stationarity condition
\begin{equation}\label{eq:subproblem_stationarity}
    \norm{ \nabla_x \AugLag_{\upsilon_k}(x^k,\hat{\mu}^k,\hat{\lambda}^k) } \leq \varepsilon_k
\end{equation}
for the next iterate.
Both criteria may be referred to as demanding that $x^k$ is an $\varepsilon_k$-approximate solution of \eqref{eq:subproblem}
in the following.

After tackling the AL subproblem \eqref{eq:subproblem},
one can compute new estimates for the Lagrange multipliers $\mu^k$ and $\lambda^k$ and then update the penalty parameter $\upsilon_k$.
A common rule for the multiplier update is to set
\begin{equation}\label{eq:standard_multiplier_update}
	\mu^k = \hat{\mu}^k + [g(x^k)-z^k]/\upsilon_k
	\quad\text{and}\quad
	\lambda^k = \hat{\lambda}^k + h(x^k)/\upsilon_k
\end{equation}
where
\begin{equation}\label{eq:formula_for_z}
	z^k = \min\{ 0, g(x^k) + \upsilon_k \hat{\mu}^k \}.
\end{equation}
The specific rule \eqref{eq:standard_multiplier_update} was adopted in \cite{conn1991globally,birgin2014practical} 
and guarantees that
\[
\nabla_x \Lag(x^k,\mu^k,\lambda^k)
=
\nabla_x \AugLag_{\upsilon_k}(x^k,\hat{\mu}^k,\hat{\lambda}^k)
\]
holds, by design, for each $k\in\N$, so that stationarity for \eqref{eq:NLP} can be monitored by means of stationarity for \eqref{eq:subproblem}, see \eqref{eq:subproblem_stationarity}.

\begin{remark}\label{rem:update_of_multiplier}
	Combining \eqref{eq:standard_multiplier_update} and \eqref{eq:formula_for_z} leads to the classical formula
	\begin{align*}
		\mu^k
		&=
		\hat\mu^k + [g(x^k)-z^k]/\upsilon_k
		=
		\hat\mu^k + \max\{g(x^k)/\upsilon_k,-\hat\mu^k\}
		=
		\max\{0,\hat\mu^k + g(x^k)/\upsilon_k\}
	\end{align*}
	for the update rule of the Lagrange multipliers associated with inequality constraints.
\end{remark}

\subsection{Elastic safeguarding}\label{sec:elastic_safeguarding}

The safeguarded ALM in \cite[\S 4.1]{birgin2014practical} demands all multiplier estimates 
to lie in prescribed nonempty, compact safeguarding sets $\safeYY^g\subseteq\R^m$ and $\safeYY^h\subseteq\R^\ell$.
For this purpose, one typically exploits
\begin{equation}\label{eq:natural_choice_safeYY}
	\safeYY^g = [0,\mu_{\max}],
	\quad
	\safeYY^h = [\lambda_{\min},\lambda_{\max}],
\end{equation}
where $\mu_{\max}\in\R^m$ and $\lambda_{\min},\lambda_{\max}\in\R^\ell$ satisfy $\mu_{\max}>0$, $\lambda_{\min}<0$, and $\lambda_{\max}>0$
in componentwise fashion,
but other choices might be reasonable as well.
Better behavior can be expected when the safeguarding sets $\safeYY^g$ and $\safeYY^h$ are ``large enough'' to cover a Lagrange multiplier (in case of existence).
The pseudocode of a prototypical safeguarded ALM can be found in \cref{alg:ALM_rigid_safeguard}.

\begin{algorithm2e}[htb]
	\DontPrintSemicolon
	\KwData{$\upsilon_0>0$; nonempty, compact sets $\safeYY^g\subseteq\R^m$ and $\safeYY^h\subseteq\R^\ell$; 
	$\tau,\gamma\in (0,1)$\;}
	Set $V_{-1}\gets+\infty$.\;
	\For{$k = 0,1,2\ldots$}{
		Select $\hat{\mu}^k\in \safeYY^g$ and $\hat{\lambda}^k\in \safeYY^h$.\label{step:ALM_rigid_safeguard:ysafe}\;
		Find an $\varepsilon_k$-approximate solution $x^k\in \R^n$ of \eqref{eq:subproblem}.\label{step:ALM_rigid_safeguard:subproblem}\;
		Set 
		\label{step:ALM_rigid_safeguard:z}%
        $z^k \gets \min\{0,g(x^k)+\upsilon_k\hat{\mu}^k\}$
        and
		$V_k \gets \max\{\norm{g(x^k)-z^k}, \norm{h(x^k)}\}$.\;
		Set 
		\label{step:ALM_rigid_safeguard:y}%
		$\mu^k \gets \hat{\mu}^k + [g(x^k) - z^k]/\upsilon_k$
        and
		$\lambda^k \gets \hat{\lambda}^k + h(x^k)/\upsilon_k$.\;
		\If{$V_k \leq \tau V_{k-1}$%
			\label{step:ALM_rigid_safeguard:check_V}}{%
			Set $\upsilon_{k+1} \gets \upsilon_k$.%
			\label{step:ALM_rigid_safeguard:keep_penalty_param}}
		\Else{%
			Set $\upsilon_{k+1} \gets \gamma \upsilon_k$.%
			\label{step:ALM_rigid_safeguard:change_penalty_param}}
	}
	\caption{Safeguarded ALM for \eqref{eq:NLP} with rigid safeguard.}
	\label{alg:ALM_rigid_safeguard}
\end{algorithm2e}

In order to avoid a deliberate choice of $\safeYY^g$ and $\safeYY^h$, since the Lagrange multiplier set is usually unknown a priori, and to make the AL scheme more adaptive, we consider
\cref{alg:ALM_elastic_safeguard}, where the safeguarding sets are allowed to grow.
This scheme is called ``elastic'' as opposed to the ``rigid'' strategy of \cref{alg:ALM_rigid_safeguard} with prefixed safeguarding bounds.
We note that the nonempty, compact safeguarding sets $\safeYY^g$ and $\safeYY^h$ in \cref{alg:ALM_elastic_safeguard} 
are rather arbitrary, but a natural choice would again be given by \eqref{eq:natural_choice_safeYY}.
In \cref{alg:ALM_elastic_safeguard}, the changes relative to \cref{alg:ALM_rigid_safeguard} are highlighted in \textcolor{cb2orange}{orange}.

\begin{algorithm2e}[htb]
	\DontPrintSemicolon
	\KwData{$\upsilon_0>0$; nonempty, compact sets $\safeYY^g\subseteq\R^m$ and $\safeYY^h\subseteq\R^\ell$; 
	$\tau,\gamma\in (0,1)$; {\color{cb2orange}$\eta\in(\gamma,\sqrt{\gamma})$}\;}
	Set $V_{-1}\gets+\infty$ and {\color{cb2orange}$\varrho_0 \gets 1$}.\;
	\For{$k = 0,1,2\ldots$}{
		Select $\hat{\mu}^k\in {\color{cb2orange}\varrho_k} \safeYY^g$ and $\hat{\lambda}^k\in{\color{cb2orange}\varrho_k} \safeYY^h$.\label{step:ALM_elastic_safeguard:ysafe}\;
		Find an $\varepsilon_k$-approximate solution $x^k\in \R^n$ of \eqref{eq:subproblem}.\label{step:ALM_elastic_safeguard:subproblem}\;
		Set 
		\label{step:ALM_elastic_safeguard:z}%
        $z^k \gets \min\{0,g(x^k)+\upsilon_k\hat{\mu}^k\}$
        and
		$V_k \gets \max\{\norm{g(x^k)-z^k}, \norm{h(x^k)}\}$.\;
		Set 
		\label{step:ALM_elastic_safeguard:y}%
		$\mu^k \gets \hat{\mu}^k + [g(x^k) - z^k]/\upsilon_k$
        and
		$\lambda^k \gets \hat{\lambda}^k + h(x^k)/\upsilon_k$.\;
		\If{$V_k \leq \tau V_{k-1}$%
			\label{step:ALM_elastic_safeguard:check_V}}{%
			Set $\upsilon_{k+1} \gets \upsilon_k$ and {\color{cb2orange}$\varrho_{k+1} \gets \varrho_k$}.%
			\label{step:ALM_elastic_safeguard:keep_penalty_param}}
		\Else{%
			Set $\upsilon_{k+1} \gets \gamma \upsilon_k$ and {\color{cb2orange}$\varrho_{k+1} \gets \eta \varrho_k / \gamma$}.%
			\label{step:ALM_elastic_safeguard:change_penalty_param}}
	}
	\caption{Safeguarded ALM for \eqref{eq:NLP} with {\color{cb2orange}elastic} safeguard.}
	\label{alg:ALM_elastic_safeguard}
\end{algorithm2e}

Let us motivate the precise implementation of elastic safeguarding in \cref{alg:ALM_elastic_safeguard}.
In fact, what really matters in the global convergence analysis \cite{birgin2014practical,demarchi2024local},
which addresses the rigidly safeguarded ALM stated in \cref{alg:ALM_rigid_safeguard},
is that the condition 
\begin{equation}\label{eq:elastic_safeguard_requirement_standard}
	\upsilon_k \norm{(\hat{\mu}^k, \hat{\lambda}^k)} \to 0
	\quad\text{as}\quad
	\upsilon_k\downarrow 0 
\end{equation}
remains guaranteed.
Indeed, we can work with elastic safeguarding sets $\varrho_k \safeYY^g$ and $\varrho_k \safeYY^h$,
which grow with scaling factor $\varrho_k > 0$, 
as long as $\upsilon_k\varrho_k\downarrow 0$ whenever $\upsilon_k\downarrow 0$.
This observation has been incorporated into \cref{alg:ALM_elastic_safeguard}.
The particular update at \cref{step:ALM_elastic_safeguard:change_penalty_param} 
allows to enlarge the safeguarding sets, 
rendering a debatable choice of $\safeYY^g$ and $\safeYY^h$ unnecessary, 
while perpetuating control over the terms $\upsilon_k \hat{\mu}^k$ and $\upsilon_k \hat{\lambda}^k$.
In fact, by selecting $\eta \in (\gamma,1)$, the scaling factor $\varrho_{k+1} > \varrho_k$ increases 
but the product $\upsilon_{k+1} \varrho_{k+1} = \eta \upsilon_k \varrho_k < \upsilon_k \varrho_k$ decreases. 
For $\eta = \gamma$, which is excluded in \cref{alg:ALM_elastic_safeguard}, 
the safeguarding mechanism in \cref{alg:ALM_elastic_safeguard} would coincide with the rigid one from \cref{alg:ALM_rigid_safeguard},
since $\varrho_{k+1}=\varrho_k$ for all $k\in\N$.
Furthermore, with $\eta < \gamma$, which is also prohibited in \cref{alg:ALM_elastic_safeguard}, 
the scaling factor $\varrho_{k+1}$ would decrease at \cref{step:ALM_elastic_safeguard:change_penalty_param}, 
effectively shrinking the safeguarding sets instead of enlarging them.

Let us also note that choosing $\eta<\sqrt{\gamma}$, \cref{step:ALM_elastic_safeguard:change_penalty_param} guarantees
$\upsilon_{k+1} \varrho_{k+1}^2 = \nicefrac{\eta^2}{\gamma}\, \upsilon_k \varrho_k^2 < \upsilon_k \varrho_k^2$, so that we even have
\begin{equation}\label{eq:elastic_safeguard_requirement_enhanced}
	\upsilon_k \norm{(\hat{\mu}^k, \hat{\lambda}^k)}^2 \to 0
	\quad\text{as}\quad
	\upsilon_k\downarrow 0 .
\end{equation}
The latter will be exploited in some but not all of the results within our global convergence analysis.

\section{Convergence analysis}\label{sec:conv_ana}

In this section, 
we analyze the global convergence properties of \cref{alg:ALM_elastic_safeguard}.
As we will see, we are in position to recover the results from \cite{birgin2014practical}
for the rigidly safeguarded ALM.
We study two situations separately.
In \cref{sec:conv_ana_approx_glob}, 
we suppose that the iterates computed in \cref{step:ALM_elastic_safeguard:subproblem}
satisfy the approximate global optimality condition stated in \eqref{eq:approximate_minimality}.
\cref{sec:conv_ana_approx_stat} is concerned with the setting where all iterates
obey the approximate stationarity condition \eqref{eq:subproblem_stationarity}.
\cref{sec:CQs} puts some special emphasis on the situation where qualification conditions
are employed to guarantee feasibility and stationarity of primal accumulation points.
The proofs presented here follow similar ones for the rigidly safeguarded ALM
with only minor deviations.
Nevertheless, they are presented in order to illustrate where the safeguarding properties
\eqref{eq:elastic_safeguard_requirement_standard} and \eqref{eq:elastic_safeguard_requirement_enhanced}
are required.

To start,
let us state a general result which presents sufficient criteria for the feasibility of primal accumulation points
associated with \cref{alg:ALM_elastic_safeguard}.
It neither requires \eqref{eq:approximate_minimality} nor \eqref{eq:subproblem_stationarity} to hold for all $k\in\N$,
and is an adaptation of \cite[Ex.~4.12]{birgin2014practical}.

\begin{lemma}\label{lem:_feas}
	Let $\{x^k\}$ be a sequence generated by \cref{alg:ALM_elastic_safeguard}.
	Furthermore, let one of the following conditions be valid:
	\begin{enumerate}[label=(\roman*)]
		\item\label{item:gamma_uniformly_pos}
			$\{\upsilon_k\}$ remains bounded away from zero;
		\item\label{item:AugLag_bounded}
			$\{\AugLag_{\upsilon_k}(x^k,\hat\mu^k,\hat\lambda^k)\}$ is upper bounded.
	\end{enumerate}
	Then each accumulation point of $\{x^k\}$ is feasible for \eqref{eq:NLP}.	
\end{lemma}
\begin{proof}
	Let $\bar x\in\R^n$ be an accumulation point of $\{x^k\}$,
	and pick an infinite index set $K\subseteq\N$ such that $x^k\to_K\bar x$.
	
	Assume that condition \ref{item:gamma_uniformly_pos} is valid.
	Then, for all sufficiently large $k\in\N$,
	$V_k\leq\tau V_{k-1}$ holds by construction of \cref{alg:ALM_elastic_safeguard}.
	Hence, $V_k\to 0$ follows, and this yields the convergences
	$\max\{g(x^k),-\upsilon_k\hat\mu^k\}\to 0$ and $h(x^k)\to 0$.
	The continuity of $g$ and $h$ yields $g(\bar x)\leq 0$ and $h(\bar x)=0$,
	i.e., $\bar x$ is feasible for \eqref{eq:NLP}.
	
	Now let condition \ref{item:AugLag_bounded} be valid.
	Without loss of generality, we assume that $\upsilon_k\downarrow 0$ as, otherwise,
	feasibility of $\bar x$ already follows from \ref{item:gamma_uniformly_pos}.
	Condition \ref{item:AugLag_bounded} implies the existence of a constant $c>0$ such that
	$\AugLag_{\upsilon_k}(x^k,\hat\mu^k,\hat\lambda^k)\leq c$
	is valid for all $k\in\N$. The definition of the augmented Lagrangian function yields
	\[
		\frac{1}{2}\norm{\max\{0,g(x^k)+\upsilon_k\hat\mu^k\}}^2 + \frac12\norm{h(x^k) + \upsilon_k\hat\lambda^k}^2
		\leq
		\upsilon_k(c-f(x^k))
	\]
	for each $k\in\N$.
	Taking the limit $k\to_K+\infty$ in this estimate while keeping \eqref{eq:elastic_safeguard_requirement_standard},
	continuity of $f$, $g$, as well as $h$, and the boundedness of $\{f(x^k)\}_{k\in K}$ in mind,
	we find
	\[
		\frac12\norm{\max\{0,g(\bar x)\}}^2 + \frac12\norm{h(\bar x)}^2\leq 0,
	\]
	which is equivalent to $g(\bar x)\leq 0$ and $h(\bar x)=0$, i.e., $\bar x$ is feasible for \eqref{eq:NLP}.
\end{proof}

A second preliminary result addresses the multiplier sequence $\{\mu^k\}$ produced by \cref{alg:ALM_elastic_safeguard}.

\begin{lemma}\label{lem:vanishing_multipliers}
	Let $\{(x^k,\mu^k,\lambda^k)\}$ be a primal-dual sequence generated by \cref{alg:ALM_elastic_safeguard},
	let $\bar x\in\R^n$ be an accumulation point of $\{x^k\}$,
	and let $K\subseteq\N$ be an infinite index set such that $x^k\to_K\bar x$.
	Then, for each $i\in\{1,\ldots,m\}$ such that $g_i(\bar x)<0$,
	$\mu^k_i=0$ holds for all large enough $k\in K$.
\end{lemma}
\begin{proof}
	First, assume that $\{\upsilon_k\}$ remains bounded away from zero.
	As in the proof of \cref{lem:_feas}, this implies $V_k\to 0$ and, particularly,
	$\max\{g(x^k),-\upsilon_k\hat\mu^k\}\to 0$.
	Due to $g_i(\bar x)<0$, $\upsilon_k\hat\mu^k_i\to_K0$ follows.
	However, as $\{\upsilon_k\}$ remains bounded away from zero,
	$\hat\mu^k_i\to_K0$ is obtained, which yields $\hat\mu^k_i + g_i(x^k)/\upsilon_k<0$ for
	all large enough $k\in K$.
	Hence, \cref{rem:update_of_multiplier} implies $\mu^k_i=0$ for all large enough $k\in K$.
	
	Second, let us assume that $\upsilon_k\downarrow 0$.
	Then \eqref{eq:elastic_safeguard_requirement_standard} implies $\upsilon_k\hat\mu^k_i\to 0$,
	so $g_i(x^k) + \upsilon_k\hat\mu^k_i<0$ and, thus, $\hat\mu^k_i + g_i(x^k)/\upsilon_k<0$
	hold for all large enough $k\in K$.
	Again, due to \cref{rem:update_of_multiplier}, this implies $\mu^k_i=0$ for all large enough $k\in K$.
\end{proof}

\subsection{Approximate global minimization approach}\label{sec:conv_ana_approx_glob}

In this subsection, we assume that the primal iterates of \cref{alg:ALM_elastic_safeguard}
obey the approximate global optimality condition \eqref{eq:approximate_minimality}
for each $k\in\N$.

To start,
let us inspect feasibility of primal accumulation points
associated with \cref{alg:ALM_elastic_safeguard}.
This result is a counterpart of \cite[Thm~5.1]{birgin2014practical} for the rigidly safeguarded ALM.

\begin{theorem}\label{thm:approx_glob_feas}
	Let $\{\varepsilon_k\}$ be bounded, and
	let $\{x^k\}$ be a sequence generated by \cref{alg:ALM_elastic_safeguard}
	such that \eqref{eq:approximate_minimality} holds for each $k\in\N$.
	Finally, let $\bar x\in\R^n$ be an accumulation point of $\{x^k\}$.
	Then $\bar x$ is a global minimizer of
	\begin{equation}\label{eq:minimum_feasibility_problem}
		\minimize_x\quad
    	\feasmeas(x)
    	\coloneqq
    	\frac12\norm{\max\{0,g(x)\}}^2 + \frac12\norm{h(x)}^2
    	\quad\stt\quad
    	x\in\R^n.
	\end{equation}
	Particularly, if \eqref{eq:NLP} is feasible,
	then $\liminf_{k\to+\infty} V_k=0$ and $\bar x$ is feasible for \eqref{eq:NLP}.
\end{theorem}
\begin{proof}
	For the proof, we proceed by distinguishing two cases.
	
	Assume that $\{\upsilon_k\}$ stays bounded away from zero.
	Then $\bar x$ is a feasible point of \eqref{eq:NLP} by \cref{lem:_feas}
	and, trivially, a global minimizer of \eqref{eq:minimum_feasibility_problem}.
	
	Next, we assume that $\upsilon_k\downarrow 0$.
	Pick an arbitrary point $x\in\R^n$.
	For each $k\in\N$, \eqref{eq:approximate_minimality} and the definition 
	of the augmented Lagrangian function yield
	\begin{align*}
		\upsilon_k\,f(x^k) &+ \frac12\norm{\max\{0,g(x^k)+\upsilon_k\hat{\mu}^k\}}^2 + \frac12\norm{h(x^k) + \upsilon_k\hat{\lambda}^k}^2
		\\
		&\leq 
		\upsilon_k\,f(x) + \frac12\norm{\max\{0,g(x)+\upsilon_k\hat{\mu}^k\}}^2 + \frac12\norm{h(x) + \upsilon_k\hat{\lambda}^k}^2
		+ 
		\upsilon_k\,\varepsilon_k
		.
	\end{align*}
	Recalling the safeguarding property \eqref{eq:elastic_safeguard_requirement_standard},
	we know $\upsilon_k\hat{\mu}^k\to 0$ and $\upsilon_k\hat{\lambda}^k\to 0$,
	and the boundedness of $\{\varepsilon_k\}$ yields $\upsilon_k\,\varepsilon_k\to 0$.
	Pick an infinite set $K\subseteq\N$ such that $x^k\to_K\bar x$.
	Taking the limit $k\to_K+\infty$ in the above estimate
	while respecting the continuity of $f$, $g$, and $h$,
	we end up with
	\[
		\frac12\norm{\max\{0,g(\bar x)\}}^2 + \frac12\norm{h(\bar x)}^2
		\leq
		\frac12\norm{\max\{0,g(x)\}}^2 + \frac12\norm{h(x)}^2,
	\]
	i.e., $\bar x$ is a global minimizer of \eqref{eq:minimum_feasibility_problem}.
	To show the final statement, let \eqref{eq:NLP} be feasible.
	Then, due to the above arguments, $\bar x$ must be feasible for \eqref{eq:NLP}.
	Recalling \cref{step:ALM_elastic_safeguard:z} of \cref{alg:ALM_elastic_safeguard},
	we find that, for each $k\in\N$, $V_k = \max\{\norm{\max\{g(x^k),-\upsilon_k\hat{\mu}^k\}}, \norm{h(x^k)}\}$ is valid.
	Exploiting the convergences $x^k\to_K\bar x$ and \eqref{eq:elastic_safeguard_requirement_standard}
	as well as feasibility of $\bar x$,
	taking the limit $k\to_K+\infty$ yields $V_k\to_K0$, i.e., $\liminf_{k\to+\infty}V_k=0$.
\end{proof}

Our next result, a counterpart of \cite[Thm~5.2]{birgin2014practical},
proves global optimality of primal accumulation points associated with \cref{alg:ALM_elastic_safeguard}.
For its proof, which adapts the one of \cite[Thm~4.3]{demarchi2024local}, 
we will make use of the following lemma,
see, e.g., \cite[Lem.~3.2]{demarchi2024local} 
or \cite[Lem.~2.6]{KanzowKraemerMehlitzWachsmuthWerner2025}.

\begin{lemma}\label{lem:AL_for_feasible_points}
	Let \eqref{eq:NLP} be feasible.
	Fix $\mu\in\R^m$, $\lambda\in\R^\ell$, $\upsilon>0$, and a feasible point $x\in\R^n$ of \eqref{eq:NLP}.
	Then we have
	\[
		\AugLag_{\upsilon}(x,\mu,\lambda)
		\leq
		f(x) + \frac{\upsilon}{2}\bigl(\norm{\mu}^2 + \norm{\lambda}^2\bigr).
	\]
\end{lemma}
\begin{proof}
	Using the definition of the augmented Lagrangian and $h(x)=0$, we find
	\[
		\AugLag_{\upsilon}(x,\mu,\lambda)
		=
		f(x) + \frac1{2\upsilon}\norm{\max\{0,g(x) + \upsilon\mu\}}^2 + \frac{\upsilon}{2}\norm{\lambda}^2.
	\]
	Hence, in order to prove the assertion, it is enough to show
	\[
		\norm{\max\{0,\mu + g(x)/\upsilon\}}^2 \leq \norm{\mu}^2.
	\]
	The latter clearly holds if
	\begin{equation}\label{eq:some_estimate}
		{\max}^2\{0,\mu_i+g_i(x)/\upsilon\}\leq\mu_i^2
	\end{equation}
	is valid for all $i\in\{1,\ldots,m\}$, and this will be shown in the remainder of this proof.
	Let $i\in\{1,\ldots,m\}$ be fixed.
	If $\mu_i + g_i(x)/\upsilon\leq 0$ is valid, \eqref{eq:some_estimate} follows trivially.
	If $\mu_i +  g_i(x)/\upsilon> 0$ is true, $g_i(x)\leq 0$ implies
	$0<\mu_i + g_i(x)/\upsilon\leq \mu_i$, and monotonicity of the square on the positive real line verifies \eqref{eq:some_estimate}.
\end{proof}

\begin{theorem}\label{thm:approx_glob_opt}
	Let \eqref{eq:NLP} be feasible,
	let $\{\varepsilon_k\}$ be a null sequence, and
	let $\{x^k\}$ be a sequence generated by \cref{alg:ALM_elastic_safeguard}
	such that \eqref{eq:approximate_minimality} holds for each $k\in\N$.
	Finally, let $\bar x\in\R^n$ be an accumulation point of $\{x^k\}$.
	Then $\bar x$ is a global minimizer of \eqref{eq:NLP}.
\end{theorem}
\begin{proof}
	To start, observe that $\bar x$ is a feasible point of \eqref{eq:NLP}
	according to \cref{thm:approx_glob_feas}.
	Throughout, let $K\subseteq\N$ be an infinite index set such that $x^k\to_K\bar x$.
	For some feasible point $x\in\R^n$ of \eqref{eq:NLP}
	and arbitrary $k\in\N$, \eqref{eq:approximate_minimality} and \cref{lem:AL_for_feasible_points} guarantee
	\begin{equation}\label{eq:estimate_glob_opt}
		\AugLag_{\upsilon_k}(x^k,\hat\mu^k,\hat\lambda^k)
		-
		\frac{\upsilon_k}{2}\bigl(\norm{\hat{\mu}^k}^2 + \norm{\hat{\lambda}^k}^2\bigr)
		\leq
		f(x) + \varepsilon_k.
	\end{equation}
	Now, we proceed by a distinction of cases.
	
	First, assume that $\{\upsilon_k\}$ remains bounded away from zero.
	As in the proof of \cref{thm:approx_glob_feas}, this implies $V_k\to 0$,
	i.e., $g(x^k)-z^k\to 0$ and $h(x^k)\to 0$.
	Furthermore, $\{\varrho_k\}$ remains bounded by construction,
	and so do $\{\hat\mu^k\}$ and $\{\hat\lambda^k\}$.
	We rewrite \eqref{eq:estimate_glob_opt} according to
	\[
		f(x^k) 
		+ 
		\frac{1}{2\upsilon_k}\norm{g(x^k)-z^k}^2 + \langle\hat\mu^k,g(x^k)-z^k\rangle
		+
		\frac{1}{2\upsilon_k}\norm{h(x^k)}^2 + \langle\hat\lambda^k,h(x^k)\rangle
		\leq
		f(x) + \varepsilon_k.
	\]
	Recalling that $\{\upsilon_k\}$ remains bounded away from zero
	while $\varepsilon_k\searrow 0$,
	we can take the limit $k\to_K+\infty$ in order to find $f(\bar x)\leq f(x)$,
	owing to feasibility of $\bar{x}$ and continuity of $f$, $g$, and $h$.
	The claim follows as $x\in\R^n$ is an arbitrary feasible point of \eqref{eq:NLP}.	
	
	Second, assume that $\upsilon_k\downarrow 0$.
	Then the definition of the augmented Lagrangian and \eqref{eq:estimate_glob_opt} imply validity of
	\[
		f(x^k) - \frac{\upsilon_k}{2}\bigl(\norm{\hat{\mu}^k}^2 + \norm{\hat{\lambda}^k}^2\bigr)
		\leq
		f(x) + \varepsilon_k
	\]
	for each $k\in\N$. Taking the limit $k\to_K+\infty$ while recalling $\varepsilon_k\searrow 0$, \eqref{eq:elastic_safeguard_requirement_enhanced},
	and the continuity of $f$, we find $f(\bar x)\leq f(x)$.
	Since $x\in\R^n$ was an arbitrary feasible point of \eqref{eq:NLP},
	the assertion is validated.
\end{proof}

Let us note that the proof of \cref{thm:approx_glob_opt} does not exploit a componentwise
sign of $\{\hat\mu^k\}$ and, thus, in principle, allows $\safeYY^g$ to be quite arbitrary
in \cref{alg:ALM_elastic_safeguard}. This is in contrast to the proof of
\cite[Thm~5.2]{birgin2014practical} which exploits $\{\hat\mu^k\}\subseteq\R^m_+$, i.e.,
$\safeYY^g\subseteq\R^m_+$.

\subsection{Approximate stationarity approach}\label{sec:conv_ana_approx_stat}

In this subsection, we consider the situation where the primal iterates 
of \cref{alg:ALM_elastic_safeguard} satisfy \eqref{eq:subproblem_stationarity}
for each $k\in\N$. Typically, this condition can be guaranteed by invoking
standard solvers for the treatment of the unconstrained subproblems \eqref{eq:subproblem},
such as gradient descent or Newton methods.

We start our analysis by investigating feasibility of primal accumulation
points associated with \eqref{eq:NLP}. 
Our first result is an analogue of \cref{thm:approx_glob_feas}
and a counterpart of \cite[Thm~6.3]{birgin2014practical}.

\begin{theorem}\label{thm:approx_stat_feas_stat}
	Let $\{\varepsilon_k\}$ be bounded, and
	let $\{x^k\}$ be a sequence generated by \cref{alg:ALM_elastic_safeguard}
	such that \eqref{eq:subproblem_stationarity} holds for each $k\in\N$.
	Finally, let $\bar x\in\R^n$ be an accumulation point of $\{x^k\}$.
	Then $\bar x$ is a stationary point of the feasibility problem \eqref{eq:minimum_feasibility_problem}.
\end{theorem}
\begin{proof}
	Throughout, let $K\subseteq\N$ be an infinite index set such that $x^k\to_K\bar x$.
	Again, we proceed by a distinction of cases.
	
	Let $\{\upsilon_k\}$ be bounded away from zero.
	Then, according to \cref{lem:_feas}, $\bar x$ is feasible for \eqref{eq:NLP} and, thus,
	a global minimizer of \eqref{eq:minimum_feasibility_problem}.
	Particularly, $\bar x$ is a stationary point of \eqref{eq:minimum_feasibility_problem}.
	
	Next, we assume that $\upsilon_k\downarrow 0$.
	For each $k\in\N$, due to validity of \eqref{eq:subproblem_stationarity}, 
	we find $\xi^k\in\R^n$ such that $\norm{\xi^k}\leq 1$ and
	\begin{equation}\label{eq:approx_stat}
		\varepsilon_k\,\xi^k 
		=
		\nabla f(x^k)
		+
		g'(x^k)^\top\max\left\{0,\hat\mu^k + g(x^k)/\upsilon_k\right\}
		+
		h'(x^k)^\top\left(\hat\lambda^k + h(x^k)/\upsilon_k\right)
	\end{equation}
	are valid. Multiplying the latter identity by $\upsilon_k$ shows that
	\begin{align*}
		\upsilon_k\,\varepsilon_k\,\xi^k
		=
		\upsilon_k\,\nabla f(x^k)
		+
		g'(x^k)^\top\max\{0,g(x^k) + \upsilon_k\hat\mu^k\}
		+
		h'(x^k)^\top(h(x^k) + \upsilon_k\hat\lambda^k)
	\end{align*}
	is valid for all $k\in\N$.
	Taking the limit $k\to_K+\infty$ in this identity
	while respecting the boundedness of $\{\varepsilon_k\}$ and $\{\xi^k\}$,
	\eqref{eq:elastic_safeguard_requirement_standard}, and the continuity of $f$, $g$, $h$, $g'$, and $h'$,
	we find
	\[
		0 = g'(\bar x)^\top\max\{0,g(\bar x)\} + h'(\bar x)^\top h(\bar x),
	\]
	and the latter means that $\bar x$ is stationary for \eqref{eq:minimum_feasibility_problem}.
\end{proof}

Let us recall the notion of approximately stationary points of \eqref{eq:NLP}
in the sense of \cite[Def.~3.1]{birgin2014practical}.

\begin{definition}\label{def:AKKT}
	A feasible point $\bar x\in\R^n$ of \eqref{eq:NLP} is referred to as an
	approximately stationary point of \eqref{eq:NLP} whenever there exist sequences
	$\{x^k\}\subseteq\R^n$, $\{\mu^k\}\subseteq\R^m_+$, and $\{\lambda^k\} \subseteq\R^\ell$
	such that the convergences
	\begin{align*}
		x^k&\to\bar x,
		\\
		\nabla_x\Lag(x^k,\mu^k,\lambda^k)&\to 0,
		\\
		\max\{g_i(x^k),-\mu^k_i\}&\to 0
		\quad \forall i \in \{1,\ldots,m\}
	\end{align*}
	hold.
\end{definition}

It is well known that local minimizers of \eqref{eq:NLP} are approximately stationary points,
even in the absence of any constraint qualification,
see, e.g., \cite[Thm~3.1]{birgin2014practical},
and that very mild constraint qualifications like the so-called cone-continuity property,
introduced and studied in \cite{AndreaniMartinezRamosSilva2016},
are enough to ensure that an approximately stationary point is already stationary.

Our next result, an adaptation of \cite[Thm~6.2]{birgin2014practical},
validates that primal accumulation points associated with \cref{alg:ALM_elastic_safeguard}
are approximately stationary.
Keeping the above in mind, \cref{alg:ALM_elastic_safeguard} is, thus,
in position to compute stationary points of \eqref{eq:NLP}
in the presence of mild constraint qualifications.

\begin{theorem}\label{thm:approx_stat_akkt}
	Let $\{\varepsilon_k\}$ be a null sequence,
	and let $\{x^k\}$ be a sequence generated by \cref{alg:ALM_elastic_safeguard}
	such that \eqref{eq:subproblem_stationarity} holds for each $k\in\N$.
	Finally, let $\bar x$ be an accumulation point of $\{x^k\}$ 
	which is feasible for \eqref{eq:NLP}.
	Then $\bar x$ is approximately stationary.
\end{theorem}
\begin{proof}
	Throughout, let $K\subseteq\N$ be an infinite index set such that $x^k\to_K\bar x$.
	For each $k\in\N$, validity of \eqref{eq:subproblem_stationarity} guarantees the existence
	of $\xi^k\in\R^n$ with $\norm{\xi^k}\leq 1$ and \eqref{eq:approx_stat}.
	Hence, \cref{rem:update_of_multiplier} and the multiplier update rule from \cref{step:ALM_elastic_safeguard:y}
	yield $\varepsilon_k\,\xi^k = \nabla_x\Lag(x^k,\mu^k,\lambda^k)$ for each $k\in\N$,
	so that $\varepsilon_k\searrow 0$ implies $\nabla_x\Lag(x^k,\mu^k,\lambda^k)\to 0$.
	Let us also note that $\{\mu^k\}\subseteq\R^m_+$ holds by \cref{rem:update_of_multiplier}.
	
	As we already know $x^k\to_K\bar x$, it remains to prove the convergences $\max\{g_i(x^k),-\mu^k_i\}\to_K0$
	for all $i\in\{1,\ldots,m\}$ in order to verify the claim.
	For $i\in \{1,\ldots,m\}$ such that $g_i(\bar x)=0$, 
	this convergence is immediate from $g_i(x^k)\to_K 0$ and $\{\mu^k_i\}\subseteq\R_+$.
	Hence, it remains to consider the case where $g_i(\bar x)<0$ holds.
	In this situation, according to \cref{lem:vanishing_multipliers}, $\mu^k_i=0$ is valid for all large enough $k\in K$,
	and $\max\{g_i(x^k),-\mu^k_i\}\to_K 0$ follows trivially.	
\end{proof}

\subsection{The presence of qualification conditions}\label{sec:CQs}

Though common in nonconvex optimization,
an obvious drawback of \cref{thm:approx_stat_akkt} is that feasibility of the primal accumulation point under consideration
has to be assumed. Only in that case is one guaranteed to find an approximately stationary point.
In order to ensure feasibility \emph{and} stationarity of primal accumulation points,
one typically postulates validity of qualification conditions that 
apply not only to feasible points, but to infeasible points as well.
For the rigidly safeguarded ALM,
a suitable qualification condition has been shown to be the so-called extended MFCQ,
see \cite[Thm~2]{BirginCastilloMartinez2005} for the origin of this result.
In \cite{BirginHaeserMaculanRamirez2025}, it has been shown that a weaker qualification condition
called (extended) strong quasinormality is sufficient if the standard (without safeguarding) or safeguarded ALM is considered. 
As \cref{alg:ALM_elastic_safeguard} bridges between both methods,
there is reasonable hope that the results from \cite{BirginHaeserMaculanRamirez2025} 
carry over to the situation on hand.

Let us start our considerations by recalling the definitions of extended MFCQ and extended strong quasinormality,
see \cite[Def.~2.1]{kanzow2017example} and \cite[Def.~2.1]{BirginHaeserMaculanRamirez2025}.
\begin{definition}\label{def:extended_CQ}
	Let some point $\bar x\in\R^n$ be fixed, not necessarily feasible for \eqref{eq:NLP}.
	\begin{enumerate}
	\item We say that extended MFCQ holds at $\bar x$
	if there does not exist a nonvanishing tuple $(\mu,\lambda)\in\R^m\times\R^\ell$
	such that 
	\begin{equation}\label{eq:singular_extended_multiplier}
		g'(\bar x)^\top\mu + h'(\bar x)^\top\lambda=0,\quad
		\mu\geq 0,\quad
		\mu^\top\min\{0,g(\bar x)\}=0.
	\end{equation}
	\item 
	We say that extended strong quasinormality holds at $\bar x$
	if there do not exist a nonvanishing tuple $(\mu,\lambda)\in\R^m\times\R^\ell$
	and, for each $i\in\{1,\ldots,m\}$ such that $\mu_i\neq 0$ and each $j\in\{1,\ldots,\ell\}$ such that $\lambda_j\neq 0$,
	sequences $\{y^{i,k}\}\subseteq\R^n$ and $\{z^{j,k}\}\subseteq\R^n$
	such that $y^{i,k}\to\bar x$ and $z^{j,k}\to\bar x$ as $k\to+\infty$, $g_i(y^{i,k})>0$ and $\lambda_j\,h_j(z^{j,k})>0$
	for all $k\in\N$, and \eqref{eq:singular_extended_multiplier}.
	\end{enumerate}
\end{definition}

Note that, in contrast to \cite{BirginHaeserMaculanRamirez2025}, 
we use the term \emph{extended} strong quasinormality here as, even for feasible points,
extended strong quasinormality is more restrictive than just quasinormality.

Our first result,
which is an adaptation of \cite[Thm~2.4]{kanzow2017example},
investigates an accumulation point $\bar x\in\R^n$ of a sequence $\{x^k\}$ generated by \cref{alg:ALM_elastic_safeguard} where extended MFCQ holds. 

\begin{theorem}\label{thm:approx_stat_kkt_EMFCQ}
	Let $\{\varepsilon_k\}$ be a null sequence, and let $\{(x^k,\mu^k,\lambda^k)\}$ be a primal-dual sequence generated by \cref{alg:ALM_elastic_safeguard}
	such that \eqref{eq:subproblem_stationarity} holds for each $k\in\N$.
	Let $\bar x\in\R^n$ be an accumulation point of $\{x^k\}$ where extended MFCQ is valid,
	and let $K\subseteq\N$ be an infinite index set such that $x^k\to_K\bar x$.
	Then $\{(\mu^k,\lambda^k)\}_{k\in K}$ is bounded.
	Furthermore, $\bar x$ is stationary for \eqref{eq:NLP}, 
	and each accumulation point of $\{(\mu^k,\lambda^k)\}_{k\in K}$
	is a multiplier certificating stationarity of $\bar x$.
\end{theorem}
\begin{proof}
	As in the proof of \cref{thm:approx_stat_akkt}, we find $\nabla_x\Lag(x^k,\mu^k,\lambda^k)\to 0$.
	Assuming that $\{(\mu^k,\lambda^k)\}_{k\in K}$ is not bounded,
	we may suppose $\norm{\mu^k} + \norm{\lambda^k}\to_K+\infty$ without loss of generality. 
	Let us define $\{(\tilde\mu^k,\tilde\lambda^k)\}_{k\in K}\subseteq\R^m\times\R^\ell$ by
	\begin{equation}\label{eq:normalize_multiplier}
		\forall k\in K\colon\quad
		(\tilde\mu^k,\tilde\lambda^k)
		\coloneqq
		\frac{(\mu^k,\lambda^k)}{\norm{\mu^k} + \norm{\lambda^k}}.
	\end{equation}
	Along yet another subsequence (without relabeling), $\{(\tilde\mu^k,\tilde\lambda^k)\}_{k\in K}$
	converges to a nonvanishing tuple $(\tilde \mu,\tilde\lambda)\in\R^m\times\R^\ell$.
	By construction, we have $\tilde\mu\geq 0$, see \cref{rem:update_of_multiplier},
	and for each $i\in\{1,\ldots,m\}$ such that $g_i(\bar x)<0$,
	$\tilde\mu_i=0$ follows from \cref{lem:vanishing_multipliers}, i.e., $\tilde\mu^\top\min\{0,g(\bar x)\}=0$ is obtained.
	From $\nabla_x\Lag(x^k,\mu^k,\lambda^k)\to 0$ and $\norm{\mu^k} + \norm{\lambda^k}\to_K+\infty$ we find $g'(\bar x)^\top\tilde\mu + h'(\bar x)^\top\tilde\lambda=0$.
	Recalling that $(\tilde\mu,\tilde\lambda)$ does not vanish,
	extended MFCQ is violated at $\bar x$, contradicting the assumptions of the theorem.
	Hence, $\{(\mu^k,\lambda^k)\}_{k\in K}$ is bounded.
	
	Let $(\mu,\lambda)\in\R^m\times\R^\ell$ be an accumulation point of $\{(\mu^k,\lambda^k)\}_{k\in K}$,
	and let us assume without loss of generality that the convergences $\mu^k\to_K\mu$ and $\lambda^k\to_K\lambda$ hold.
	Clearly, $\nabla_x\Lag(\bar x,\lambda,\mu)=0$ follows from $\nabla_x\Lag(x^k,\mu^k,\lambda^k)\to 0$.
	Furthermore, similarly as above,
	$\mu\geq 0$ and $\mu^\top\min\{0,g(\bar x)\}=0$ are obtained from \cref{rem:update_of_multiplier}
	and \cref{lem:vanishing_multipliers}.
	It remains to prove that $\bar x$ is feasible.
	Whenever $\{\upsilon_k\}$ remains bounded away from zero,
	this is an immediate consequence of \cref{lem:_feas}.
	Hence, let us assume that $\upsilon_k\downarrow 0$.
	From \cref{rem:update_of_multiplier} and \cref{step:ALM_elastic_safeguard:y} of \cref{alg:ALM_elastic_safeguard}
	we find
	\begin{equation}\label{eq:helpful_representation_of_update}
		\upsilon_k\mu^k = \max\{0,\upsilon_k\hat \mu^k + g(x^k)\},
		\quad
		\upsilon_k\lambda^k = \upsilon_k\hat\lambda^k + h(x^k)
	\end{equation}
	for each $k\in K$.
	Due to boundedness of $\{(\mu^k,\lambda^k)\}_{k\in K}$,
	we have $\upsilon_k\mu^k\to_K 0 $ and $\upsilon_k\lambda^k\to_K 0$.
	Furthermore, the convergences $\upsilon_k\hat\mu^k\to 0$ and $\upsilon_k\hat\lambda^k\to 0$
	follow from \eqref{eq:elastic_safeguard_requirement_standard}.
	Thus, taking the limit $k\to_K+\infty$ in \eqref{eq:helpful_representation_of_update}
	while respecting continuity of $g$ and $h$,
	we end up with $0 = \max\{0,g(\bar x)\}$ and $0=h(\bar x)$,
	i.e., $\bar x$ is feasible for \eqref{eq:NLP}.
	Hence, $\bar x$ is a stationary point of \eqref{eq:NLP} with associated multiplier $(\mu,\lambda)$.
\end{proof}

Our second result is concerned with a primal accumulation point $\bar x\in\R^n$ 
of some primal-dual sequence $\{(x^k,\mu^k,\lambda^k)\}$ generated by \cref{alg:ALM_elastic_safeguard}
where extended strong quasinormality holds. In this situation, an additional assumption is needed 
to verify boundedness of the associated Lagrange multipliers $(\mu^k,\lambda^k)$.

\begin{assum}\label{ass:unbounded_multipliers_give_infeasibility_along_iterates}
	Let $\{(x^k,\mu^k,\lambda^k)\}$ be a primal-dual sequence generated by \cref{alg:ALM_elastic_safeguard}
	such that \eqref{eq:subproblem_stationarity} holds for each $k\in\N$.
	Let $K\subseteq\N$ be a given infinite index set such that $x^k\to_K\bar x$.
	For each infinite index set $L\subseteq K$, we assume that the following conditions hold.
	\begin{enumerate}
		\item For each $i\in\{1,\ldots,m\}$, whenever $\mu^k_i\to_L+\infty$, 
			then there is an infinite index set $L_i\subseteq K$ such that $g_i(x^k)>0$ holds for all $k\in L_i$.
		\item For each $j\in\{1,\ldots,\ell\}$, whenever $\lambda^k_j\to_L\pm\infty$, 
			then there is an infinite index set $L_j\subseteq L$ such that $\lambda^k_j\,h_j(x^k)>0$ holds for all $k\in L_j$.
	\end{enumerate}
\end{assum}

Now, we are in position to prove an adaptation of \cite[Thm~2.1]{BirginHaeserMaculanRamirez2025}
that applies to \cref{alg:ALM_elastic_safeguard}.

\begin{theorem}\label{thm:approx_stat_kkt_EQN}
	Let $\{\varepsilon_k\}$ be a null sequence, and let $\{(x^k,\mu^k,\lambda^k)\}_{k\in K}$ be a primal-dual sequence generated by \cref{alg:ALM_elastic_safeguard}
	such that \eqref{eq:subproblem_stationarity} holds for each $k\in\N$.
	Let $\bar x\in\R^n$ be an accumulation point of $\{x^k\}$ where extended strong quasinormality is valid.
	Let $K\subseteq\N$ be an infinite index set such that $x^k\to\bar x$
	and assume that \cref{ass:unbounded_multipliers_give_infeasibility_along_iterates}
	holds for this index set $K$.
	Then $\{(\mu^k,\lambda^k)\}_{k\in K}$ is bounded.
	Furthermore, $\bar x$ is stationary for \eqref{eq:NLP},
	and each accumulation point of $\{(\mu^k,\lambda^k)\}_{k\in K}$
	is a multiplier certificating stationarity of $\bar x$.
\end{theorem}
\begin{proof}
	As before, we have $\nabla_x\Lag(x^k,\mu^k,\lambda^k)\to 0$ by construction.
	Assuming that $\{(\mu^k,\lambda^k)\}$ is not bounded,
	we may proceed as in the proof of \cref{thm:approx_stat_kkt_EMFCQ}
	in order to find, via construction \eqref{eq:normalize_multiplier}, 
	a nonvanishing tuple $(\tilde\mu,\tilde\lambda)\in\R^m\times\R^\ell$
	such that $\tilde\mu\geq 0$, $\tilde\mu^\top\min\{0,g(\bar x)\}=0$, 
	and $g'(\bar x)^\top\tilde\mu + h'(\bar x)^\top\tilde\lambda=0$.
	Throughout, let $L\subseteq K$ be an infinite index set 
	such that $\norm{\mu^k} + \norm{\lambda^k}\to_L+\infty$
	and $(\tilde\mu^k,\tilde\lambda^k)\to_L(\tilde\mu,\tilde\lambda)$.
	
	For each $i\in\{1,\ldots,m\}$ such that $\tilde\mu_i\neq 0$,
	we know $\mu^k_i\to_L +\infty$.
	Hence, \cref{ass:unbounded_multipliers_give_infeasibility_along_iterates}
	equips us with an infinite index set $L_i\subseteq K$
	such that $g_i(x^k)>0$ holds for all $k\in L_i$.
	We define $y^{i,k} \coloneqq x^k$ for each $k\in L_i$ and observe that $y^{i,k}\to_{L_i}\bar x$.
	
	For each $j\in\{1,\ldots,\ell\}$ such that $\tilde\lambda_j\neq 0$,
	we know $\lambda^k_j\to_L\pm\infty$,
	and $\sgn(\lambda^k_j)=\sgn(\tilde\lambda_j)$ holds for all large enough $k\in L$.
	Hence, \cref{ass:unbounded_multipliers_give_infeasibility_along_iterates}
	equips us with an infinite index set $L_j\subseteq L$
	such that $\lambda^k_j h_j(x^k)>0$ holds for all $k\in L_j$.
	Hence, for large enough $k\in L_j$, we have $\tilde\lambda_j h_j(x^k)>0$.
	We define $z^{j,k} \coloneqq x^k$ for each $k\in L_j$ and observe that $z^{j,k}\to_{L_j}\bar x$.
	
	By construction, extended strong quasinormality is violated at $\bar x$, contradicting the assumptions
	of the theorem. Hence, $\{(\mu^k,\lambda^k)\}_{k\in K}$ is bounded.
	
	The remainder of the proof parallels the one of \cref{thm:approx_stat_kkt_EMFCQ}.
\end{proof}

\begin{remark}\label{rem:regaring_strong_quasinormality}
	Inspecting the proof of \cref{thm:approx_stat_kkt_EQN},
	it is apparent that, for $j\in\{1,\ldots,\ell\}$ such that $\tilde\lambda_j\neq 0$
	(i.e., $\lambda^k_j\to\pm\infty$),
	\cref{ass:unbounded_multipliers_give_infeasibility_along_iterates} requires
	$L_j\subseteq L$, and not merely $L_j\subseteq K$ as might be expected from the treatment of the inequality constraints.
	This stronger condition is needed to ensure validity of the sign conditions required in the definition of extended strong quasinormality.
\end{remark}

The upcoming example 
presents an instance of \cref{alg:ALM_elastic_safeguard} 
where \cref{ass:unbounded_multipliers_give_infeasibility_along_iterates}
holds for $K = \N$, see \cite[Thm~3.2]{BirginHaeserMaculanRamirez2025} as well,
i.e., it applies if the produced primal sequence converges as a whole.

\begin{example}\label{ex:kkt_via_EQN}
	Let us assume that the safeguarding sets $\safeYY^g$ and $\safeYY^h$ are given as in \eqref{eq:natural_choice_safeYY},
	and pick $\hat\mu^k\in\rho_k\safeYY^g$ and $\hat\lambda^k\in\rho_k\safeYY^h$ in \cref{step:ALM_elastic_safeguard:ysafe} such that
	\begin{equation}\label{eq:projection_rule_safeguarding}
		\begin{aligned}
			\hat \mu^k & = \max\{0,\min\{\mu^{k-1},\rho_k\,\mu_{\max}\}\},
			\\
			\hat \lambda^k & = \max\{-\bar\rho_k\,\lambda_{\min},\min\{\lambda^{k-1},\bar\rho_k\,\lambda_{\max}\}\},
		\end{aligned}
	\end{equation}
	where $\bar\rho_k\coloneqq\min\{\rho_k,\bar\rho\}$ is used for some given $\bar\rho>0$.
	Note that this implies that $\{\hat\lambda^k\}\subseteq\bar\rho\safeYY^h$,
	i.e., equality constraints are effectively treated with a rigid safeguard.
	
	Let $\{(x^k,\mu^k,\lambda^k)\}$ be a primal-dual sequence generated by \cref{alg:ALM_elastic_safeguard}, 
	and assume that $x^k\to\bar x$ holds for some $\bar x\in\R^n$.
	We show that \cref{ass:unbounded_multipliers_give_infeasibility_along_iterates}
	holds for $K=\N$.
	Therefore, we pick an arbitrary infinite index set $L\subseteq\N$.
	
	Pick $i\in\{1,\ldots,m\}$ such that $\mu^k_i\to_L+\infty$.
	Observe that $\mu^k_i=\hat \mu^k_i + g_i(x^k)/\upsilon_k$ holds for each large enough $k\in L$,
	see \cref{rem:update_of_multiplier}.
	From $\mu^k_i\to_{L}+\infty$ there must exist an infinite index set $L_i\subseteq\N$
	such that $\mu^k_i>\mu^{k-1}_i$ holds for all $k\in L_i$ 
	as, otherwise, $\mu^k_i\leq\mu^{k-1}_i$ holds for all large enough
	$k\in\N$, yielding boundedness of $\{\mu^k_i\}$ and, thus, a contradiction.
	Thus, for each $k\in L_i$, 
	we have $\mu^{k-1}_i < \mu^k_i  =  \hat\mu^k_i + g_i(x^k)/\upsilon_k \leq \mu^{k-1}_i + g_i(x^k)/\upsilon_k$,
	i.e., $g_i(x^k)>0$ due to $\upsilon_k>0$.

	Pick $j\in\{1,\ldots,\ell\}$ such that $\lambda^k_j\to_L\pm \infty$.
	As we have $\lambda^k_j = \hat\lambda^k_j + h_j(x^k)/\upsilon_k$ for all $k\in L$
	while $\{\hat\lambda^k_j\}_{k\in L}$ remains bounded,
	we find $|h_j(x^k)/\upsilon_k|\to_L+\infty$ and, due to $\upsilon_k>0$, $\sgn(\lambda^k_j)=\sgn(h_j(x^k))$
	for all large enough $k\in L$.
	Hence, there is an infinite index set $L_j\subseteq L$, a suitable tail of $L$, exemplary,
	such that $\lambda^k_j h_j(x^k)>0$ must be valid for all $k\in L_j$.
	
	Taking everything together, 
	we see that \cref{ass:unbounded_multipliers_give_infeasibility_along_iterates}
	holds for $K = \N$.
\end{example}

We end this section with two remarks.

\begin{remark}
	Inspecting the arguments in \cref{ex:kkt_via_EQN} that addressed the inequality constraints,
	it appears that, for $i\in\{1,\ldots,m\}$, we only need to guarantee $\hat\mu^k_i\leq\mu^{k-1}_i$
	in \cref{step:ALM_elastic_safeguard:ysafe} of \cref{alg:ALM_elastic_safeguard}, 
	which trivially holds for selection rules of type
	\[
	\forall i\in\{1,\ldots,m\}\colon\quad
	\hat \mu^k_i
	\begin{cases}
		=\mu^{k-1}_i	&	\text{if }0\leq\mu^{k-1}_i\leq\rho_k\,\mu_{\max,i},
		\\
		\in[0,\rho_k\,\mu_{\max,i}]	& \text{otherwise,}
	\end{cases}
	\]
	due to $\{\mu^k\}\subseteq\R^m_+$.
	Exemplary, one may choose $\hat\mu^k_i = 0$ if $\mu^{k-1}_i>\rho_k\,\mu_{\max,i}$, as suggested in \cite[\S 5.3]{birgin2014practical}.
\end{remark}

\begin{remark}\label{rem:}
	For equality constraints, instead, the argument in \cref{ex:kkt_via_EQN} seems to fail when projecting onto $\rho_k\safeYY^h$ (rather than $\bar{\rho}_k\safeYY^h$).
	It remains an open question whether the result holds or not when adopting the update rule 
	$\hat \lambda^k = \max\{-\rho_k\,\lambda_{\min},\min\{\lambda^{k-1},\rho_k\,\lambda_{\max}\}\}$ for the multipliers of equality constraints.
	Regardless, since equality constraints were not considered in \cite{BirginHaeserMaculanRamirez2025}, \cref{thm:approx_stat_kkt_EQN} provides a more general result.
\end{remark}

\section{Complexity analysis}\label{sec:complexity}

This section is devoted to the derivation of worst-case complexity results related to \cref{alg:ALM_elastic_safeguard},
which remain analogous to those valid for \cref{alg:ALM_rigid_safeguard}.
A complexity result for the case where $\{\upsilon_k\}$ remains bounded away from zero directly follows from \cite[Thm 3.1]{birgin2020complexity} under mild assumptions.
In general, though, \cref{alg:ALM_elastic_safeguard} may stop at an iterate $x^k$ that appears to be infeasible and, at the same time, a local minimizer of the feasibility measure.
Note that the possibility $\upsilon_k \downarrow 0$ must be considered since it necessarily takes place, for example, when the feasible region of \eqref{eq:NLP} is empty.

The following result establishes that, before the number of iterations given by \eqref{eq:complexity_outer} is reached, we necessarily find an approximate stationary point (up to preliminarily chosen tolerances) for \eqref{eq:NLP}
or we find an infeasible point that is an approximate stationary point (up to preliminarily chosen tolerances) for the feasibility problem \eqref{eq:minimum_feasibility_problem}.

To prove the result, we rely on some boundedness assumptions.
In particular, we assume that there exist constants $\cf>0$, $\cg>0$, $\ch>0$, and $\cbig>0$ such that
\begin{subequations}\label{eq:boundedness_ass}
	\begin{align}
		\norm{\nabla f(x^k)} \leq{}& \cf ,\label{cf} \\
		\norm{g^\prime(x^k)} \leq{}& \cg ,\quad
		\norm{h^\prime(x^k)} \leq{} \ch ,\label{cgh} \\
		V_k \leq{}& \cbig \label{cbig}
	\end{align}
\end{subequations}
hold along all iterations $k\in\N$ of \cref{alg:ALM_elastic_safeguard}.
Conditions \eqref{cf} and \eqref{cgh} are analogous to the assertions of \cite[Lem.~3.3]{birgin2020complexity}, 
that result from the continuity of the functions $\nabla f$, $g^\prime$, and $h^\prime$, and the assumed boundedness of $x$, $\hat{\mu}$, and $\hat{\lambda}$
therein.
Condition \eqref{cbig} was adopted also in \cite[Thm 3.1]{birgin2020complexity} for the case where $\{\upsilon_k\}$ remains bounded away from zero.
Since we have
\[
    V_k 
    = 
    \max\{ \|g(x^k)-z^k\|,\|h(x^k)\| \} 
    = 
    \max\{ \|\max\{g(x^k),-\upsilon_k\hat{\mu}^k\}\|, \|h(x^k)\| \}
\]
while \eqref{eq:elastic_safeguard_requirement_standard} holds, 
\eqref{cbig} boils down to bounded constraint violation along the iterates.
A sufficient condition to satisfy \eqref{eq:boundedness_ass} is that the iterates $\{x^k\}$ remain in a compact set $\Omega\subseteq \R^n$, for instance, by imposing lower-level box constraints of type $x\in [x_l,x_u]$
which are not augmented, see \cref{sec:AL_foundations}.

For notational convenience, let us introduce the constant $\cdiam \coloneqq \cg \extent(\safeYY^g) + \ch \extent(\safeYY^h) > 0$.

\begin{theorem} \label{thm:complexity_outer}
	Let $\epsilon,\delta > 0$ and $\deltalow \in (0,\delta]$ be given,
	let $\{\varepsilon_k\}$ be a null sequence,
	and let $\{(x^k,\mu^k,\lambda^k)\}$ be a primal-dual sequence generated by \cref{alg:ALM_elastic_safeguard}
	such that \eqref{eq:subproblem_stationarity} holds for each $k\in\N$.
	Assume that $\Ninnertol(\epsilon, \deltalow)$ is such that
	$\varepsilon_k \leq \min\{\epsilon, \deltalow/4 \}$ for all $k \geq
	\Ninnertol(\epsilon, \deltalow)$.
	Suppose conditions \eqref{eq:boundedness_ass} hold for all $k\in\N$.
	Then, after at most
	\begin{equation} \label{eq:complexity_outer}
		\Niterout(\epsilon,\delta,\deltalow)
		\coloneqq
		\ceil{ \max\left\{ \Ninnertol(\epsilon, \deltalow), \frac{\log(\delta/\cbig)}{\log(\tau)} \Npenupdates(\delta, \deltalow) \right\} }
	\end{equation}      
	iterations,
	where
	\begin{equation} \label{eq:complexity_outer_updates}
		\Npenupdates(\delta, \deltalow)
		\coloneqq
		\ceil{ \max\left\{
			\frac{\log\left(\frac{\deltalow}{2\cdiam}\right)}{\log\eta},
			\frac{\log\left(\frac{\min\{ 1, \deltalow/(4\cf) \}}{\upsilon_0}\right)}{\log\gamma},
			\frac{\log\left(\frac{\delta}{\extent(\safeYY^g) \upsilon_0\varrho_0}\right)}{\log\eta}
		\right\} }
		,
	\end{equation}
	we obtain an iteration $k$ such that one of the following
	mutually exclusive situations is on hand.
	\begin{enumerate}[label=(\roman{*})]
		\item The multipliers $\mu^k \in \R^m_+$ and $\lambda^k \in \R^\ell$ are such that
		\begin{subequations}\label{eq:primaldual}
			\begin{align}
				\norm{\nabla f(x^k) +  g'(x^k)^\top \mu^k +  h'(x^k)^\top \lambda^k}
				\leq{}&
				\epsilon,
				\label{primaldual1}
				\\
				\norm{\max\{0,g(x^k)\}}_\infty\leq \delta,\quad \norm{h(x^k)}_\infty \leq{}& \delta,
				\label{primaldual2}
				\\
				g_i(x^k) <  - \delta \implies \mu^k_i ={}& 0 
				\quad\forall i\in\{1,\ldots,m\}.
				\label{primaldual3}
			\end{align}
		\end{subequations}
		%%%
		\item The iterate $x^k$ verifies
		\begin{subequations}\label{paradamala}
		\begin{align}
			\norm{\nabla \feasmeas(x^k)}
			\leq{}& \deltalow ,\label{paradamala1} \\
			\max\{\norm{\max\{0,g(x^k)\}}_\infty, \norm{h(x^k)}_\infty \} >{}& \delta , \label{paradamala2}
		\end{align}
		\end{subequations}
        where $\feasmeas$ is the feasibility measure defined in \eqref{eq:minimum_feasibility_problem}.
	\end{enumerate}
\end{theorem}
\begin{proof}
	We proceed in three steps, considering only iterations $k \geq \Ninnertol(\epsilon, \deltalow)$.
	First, it is shown that either \eqref{eq:primaldual} is satisfied or the number of iterations between two consecutive penalty updates (at \cref{step:ALM_elastic_safeguard:change_penalty_param}) is bounded above by $\log(\delta/\cbig) / \log(\tau)$.
	Second, we prove that a number $\Npenupdates(\delta, \deltalow)$ of penalty updates guarantees that \eqref{paradamala1} is satisfied.
	Third, we demonstrate also that either \eqref{paradamala2} or \eqref{primaldual2}--\eqref{primaldual3} must hold after $\Npenupdates(\delta, \deltalow)$ penalty updates.
	
	\medskip
	
	$\triangleright$
	After the first $\Ninnertol(\epsilon, \deltalow)$ iterations, \eqref{eq:subproblem_stationarity} and $\varepsilon_k\leq\epsilon$ guarantee that \eqref{primaldual1} is satisfied.
	As shown in \cite[Lem.~3.1]{birgin2020complexity}, 
    condition $V_k \leq \delta$ implies \eqref{primaldual2} and \eqref{primaldual3}.
	Owing to \eqref{cbig}, the condition at \cref{step:ALM_elastic_safeguard:check_V} indicates that it takes at most
	$\log(\delta/\cbig) / \log(\tau)$ consecutive iterations without penalty updates to reach $V_k \leq \delta$.
	This means that, after the first $\Ninnertol(\epsilon, \deltalow)$ iterations, either conditions \eqref{eq:primaldual} are satisfied or the difference between iterations with a penalty decrease is at most $\log(\delta/\cbig) / \log(\tau)$.
	
	\medskip	
	
	$\triangleright$
	We now show that, after the first $\Ninnertol(\epsilon, \deltalow)$ iterations, there can be at most $\Npenupdates(\delta, \deltalow)$ penalty updates before satisfying \eqref{paradamala1}.
	Let $u(k)\in\N$ denote the number of penalty updates up to the $k$-th iteration.
	The update rule at \cref{step:ALM_elastic_safeguard:change_penalty_param} 
    gives $\upsilon_k = \gamma^{u(k)}\upsilon_0$ and $\upsilon_k\varrho_k  = \eta^{u(k)} \upsilon_0\varrho_0 $.
    Then $u(k) \geq \Npenupdates(\delta, \deltalow)$ implies
	\begin{subequations}\label{eq:thresholds}
		\begin{align}
			u(k) \geq{}& \log\left(\deltalow/(2\cdiam)\right) / \log\eta , \label{ellk} \\
			\upsilon_k \leq{}& \min\{ 1, \deltalow/(4\cf) \} , \label{rho_uno_cf} \\
			\upsilon_k\varrho_k  \leq{}& \delta / \extent(\safeYY^g) , \label{rhomuma}
		\end{align}
	\end{subequations}
	where condition \eqref{ellk} follows from the first term in the $\max$ operator within \eqref{eq:complexity_outer_updates}, 
	\eqref{rho_uno_cf} from the second, and \eqref{rhomuma} from the third.
	We now proceed by showing that $k\geq \Ninnertol(\epsilon, \deltalow)$ and \eqref{eq:thresholds} yield \eqref{paradamala1}.
	
	For notational convenience, let us define, for each $k\in\N$, the perturbed feasibility measure
	\begin{equation*}
		\forall x\in\R^n\colon\quad
		\pertfeasmeas_k(x)
		\coloneqq
		\frac{1}{2} \norm{\max\{0,g(x) + \upsilon_k \hat{\mu}^k\}}^2 + \frac{1}{2} \norm{h(x) + \upsilon_k \hat{\lambda}^k}^2 .
	\end{equation*}
	Then, by \eqref{eq:subproblem_stationarity}, for all $k \geq 1$, we have that
	\[
	\varepsilon_k
	\geq
	\norm{\nabla_x \AugLag_{\upsilon_k}(x^k,\hat{\mu}^k,\hat{\lambda}^k) }
	=
	\left\|
	\nabla f(x^k) + \frac{1}{\upsilon_k} \nabla \pertfeasmeas_k(x^k)
	\right\|
	\geq
	\norm{
	\upsilon_k \nabla f(x^k) + \nabla \pertfeasmeas_k(x^k)
	}
	,
	\]
	where the last inequality follows from \eqref{rho_uno_cf} (particularly, $\upsilon_k \leq 1$), and then
	\begin{equation} \label{quiero1}
		\norm{
		\nabla \pertfeasmeas_k(x^k)
		}
		\leq
		\varepsilon_k + \cf \upsilon_k
		\leq
		\deltalow/4 + \cf \upsilon_k
		\leq
		\deltalow/2
		,
	\end{equation} 
	where the first inequality is due to \eqref{cf},
	the second to the definition of $\Ninnertol(\epsilon, \deltalow)$ and $k\geq \Ninnertol(\epsilon, \deltalow)$,
	and the third to \eqref{rho_uno_cf} (particularly, $\cf \upsilon_k \leq \deltalow/4$).
	Now we bound the difference between the gradient of $\pertfeasmeas_k$ and that of $\feasmeas$, obtaining
	\begin{align*}
		\norm{\nabla \pertfeasmeas_k(x^k)
		- \nabla \feasmeas(x^k)}
		={}&
		\norm{
		g^\prime(x^k)^\top \max\{0,g(x^k)+\upsilon_k\hat{\mu}^k\}
		+ h^\prime(x^k)^\top \upsilon_k\hat{\lambda}^k
		- g^\prime(x^k)^\top \max\{0,g(x^k)\}
		}
		\\
		\leq{}&
		\norm{g^\prime(x^k)}
		\norm{
		\max\{0,g(x^k)+\upsilon_k\hat{\mu}^k\} - \max\{0,g(x^k)\}
		}
		+ \norm{h^\prime(x^k)}\norm{ \upsilon_k\hat{\lambda}^k }
		\\
		\leq{}&
		\cg
		\norm{
		\upsilon_k\hat{\mu}^k
		}
		+
		\ch \norm{
		\upsilon_k\hat{\lambda}^k
		}
		,
	\end{align*}
	where the last estimate follows from the Lipschitzness of $s\mapsto\max\{0,s\}$ 
    with Lipschitz constant 1 and \eqref{cgh}.
	Therefore, from \eqref{quiero1} we obtain that
	\begin{equation*}
		\norm{\nabla \feasmeas(x^k)}
		\leq
		\frac{\deltalow}{2}
		+ \cg \|
		\upsilon_k\hat{\mu}^k
		\|
		+ \ch
		\|
		\upsilon_k\hat{\lambda}^k
		\|
	\end{equation*}
	for each $k$.
	By counting the number $u(k)$ of penalty updates, \cref{step:ALM_elastic_safeguard:change_penalty_param} gives
	\[
	\norm{ \upsilon_k\hat{\mu}^k }\leq \eta^{u(k)} \extent(\safeYY^g)
	,\quad
	\norm{\upsilon_k\hat{\lambda}^k } \leq \eta^{u(k)} \extent(\safeYY^h)
	,
	\]
	for each $k$, hence
	\begin{equation*}
		\norm{\nabla \feasmeas(x^k)}
		\leq
		\frac{\deltalow}{2} + \left[ \cg \extent(\safeYY^g) + \ch \extent(\safeYY^h) \right] \eta^{u(k)}
		=
		\frac{\deltalow}{2} + \cdiam \eta^{u(k)} .
	\end{equation*}
	Since \eqref{ellk} implies $\eta^{u(k)} \leq \deltalow/(2 \cdiam)$, it leads to \eqref{paradamala1}.
	
	\medskip
	
	$\triangleright$
	Finally, we prove that, after $\Npenupdates(\delta, \deltalow)$ penalty updates, if \eqref{paradamala2} does not hold, then \eqref{primaldual2} and \eqref{primaldual3} are valid.
	If \eqref{paradamala2} fails, \eqref{primaldual2} clearly holds.
	Therefore, $g_i(x^k) \leq \delta$ for all $i\in\{1,\ldots,m\}$.
	Now, if $g_i(x^k) < -\delta$, we can exploit \eqref{rhomuma} to find
    \[
        g_i(x^k) + \upsilon_k \hat{\mu}^k_i
        < 
        - \delta + \upsilon_k \hat{\mu}^k_i 
        \leq 
        - \delta + \upsilon_k \varrho_k \extent(\safeYY^g)
        <
        0, 
    \]
    so $\mu_i^k = 0$ follows from \cref{rem:update_of_multiplier}.
	Therefore, \eqref{primaldual3} is valid, too.
	
	\medskip
	
	Combining these observations, 
	we have shown that before the number of iterations given by \eqref{eq:complexity_outer} is reached --- and $\Npenupdates(\delta, \deltalow)$ penalty updates are done --- either \eqref{eq:primaldual} or \eqref{paradamala} is satisfied.
	This concludes the proof.
\end{proof}
	
We now derive a bound on the total number of iterations and function evaluations
which are required to find an iteration $k$ of \cref{alg:ALM_elastic_safeguard} such that the qualitative properties discussed in \cref{thm:complexity_outer} hold
for the primal-dual triplet $(x^k,\mu^k,\lambda^k)$. 
Therefore, we combine estimates for outer and inner iteration numbers (and function evaluations).
Estimating the outer complexity is possible via \cref{thm:complexity_outer}.
To address the inner complexity, we rely on \cite[Thm~4.3]{birgin2020complexity}, which requires that
\begin{itemize}
	\item
		the iterates $\{x^k\}$ remain in a bounded set $\Omega \subseteq \R^n$, and
	\item 
		for fixed penalty parameter and multiplier estimates, the augmented Lagrangian function has a Lipschitz continuous gradient over $\Omega$.
\end{itemize}
Then, adopting a gradient-based method for solving \eqref{eq:subproblem}, 
there exists a constant $\citerin > 0$ such that the number of iterations and function evaluations necessary to satisfy \eqref{eq:subproblem_stationarity} 
is bounded above by
\begin{equation}\label{eq:complexity_inner}
	\Niterin(\upsilon_k, \varepsilon_k) = \citerin \upsilon_k^{-2} \varepsilon_k^{-2}.
\end{equation}
Departing from \cite{birgin2020complexity}, where the safeguarding mechanism is rigid, the multiplier estimates are not necessarily bounded in \cref{alg:ALM_elastic_safeguard}.
Nevertheless, the safeguarding condition \eqref{eq:elastic_safeguard_requirement_standard} allows us to use the inner complexity bound \eqref{eq:complexity_inner} from \cite[Thm~4.3]{birgin2020complexity}.
In fact, since $\|\hat{\mu}^k\|$ and $\|\hat{\lambda}^k\|$ become smaller than $1/\upsilon_k$ as $\upsilon_k\downarrow 0$, 
the Lipschitz constant of $\nabla_x \AugLag_{\upsilon_k}(\cdot,\hat{\mu}^k,\hat{\lambda}^k)$ is asymptotically dominated by $1/\upsilon_k$.

The following result is inspired by \cite[Thm~3.6]{birgin2020complexity} and provides a total complexity bound for \cref{alg:ALM_elastic_safeguard} under analogous Lipschitz assumptions.
Moreover, to obtain a concrete bound, we specify the sequence of inner tolerances $\{\varepsilon_k\}$.

\begin{theorem}\label{thm:total_count_iterations}
	Consider the assumptions of \cref{thm:complexity_outer} and suppose, additionally, that
	\begin{itemize}
		\item there exists $\citerin>0$ such that establishing \eqref{eq:subproblem_stationarity} in \cref{step:ALM_elastic_safeguard:subproblem} of \cref{alg:ALM_elastic_safeguard} 
			requires a number of inner iterations (and function evaluations) that is at most $\Niterin(\upsilon_k, \varepsilon_k)$ given in \eqref{eq:complexity_inner},
		\item for some $\varepsilon_0 > 0$ and $\gamma_\varepsilon \in (0,1)$, $\varepsilon_k = \max\{ \min\{\epsilon, \deltalow/4 \}, \gamma_\varepsilon^k \varepsilon_0 \}$ holds for each $k\in\N$.
	\end{itemize}
	Let $\Niterout \coloneqq \Niterout(\epsilon, \delta, \deltalow)$ be the number of outer iterations defined by \eqref{eq:complexity_outer} and set $\theta \coloneqq (\gamma\gamma_\varepsilon)^{-2}$.
	Then $\Ninnertol(\epsilon, \deltalow) = \ceil{ \log(\min\{\epsilon, \deltalow/4 \}/\varepsilon_0) / \log(\gamma_\varepsilon) }$, 
	and an upper bound on the total number of inner iterations (and function evaluations) required to obtain an outer iteration $k$
	such that the qualitative properties claimed in \cref{thm:complexity_outer} hold is given by
	\begin{equation*}
		\Nitertot(\epsilon, \delta, \deltalow)
		\coloneqq
		\ceil{ \frac{\citerin}{(\upsilon_0 \varepsilon_0)^2} \frac{\theta^{\Niterout + 1}-1}{\theta-1} } .
	\end{equation*}
\end{theorem}
\begin{proof}
	The first assertion on $\Ninnertol(\epsilon, \deltalow)$ is obvious from the precise choice of $\{\varepsilon_k\}$ considered here.
	The second assertion immediately follows from 
	\begin{equation*}
		\sum_{k=0}^{\Niterout} \Niterin(\upsilon_k,\varepsilon_k)
		=
		\sum_{k=0}^{\Niterout} \frac{\citerin}{(\upsilon_k\varepsilon_k)^2}
		\leq
		\sum_{k=0}^{\Niterout} \frac{\citerin}{(\upsilon_0 \gamma^k\varepsilon_0 \gamma_\varepsilon^k)^2}
		=
		\frac{\citerin}{(\upsilon_0\varepsilon_0)^2} \sum_{k=0}^{\Niterout} \theta^k
		=
		\frac{\citerin}{(\upsilon_0 \varepsilon_0)^2} \frac{\theta^{\Niterout+1}-1}{\theta-1}
	\end{equation*}
	where the first equality is due to \eqref{eq:complexity_inner} and the inequality follows from $\upsilon_k \geq \gamma^k \upsilon_0$ and $\varepsilon_k \geq \gamma_\varepsilon^k \varepsilon_0$
	which are valid for each $k\in\N$.
\end{proof}
In \cref{thm:total_count_iterations}, note that the function evaluations $g(x^k)$ and $h(x^k)$ at \cref{step:ALM_elastic_safeguard:z,step:ALM_elastic_safeguard:y} of \cref{alg:ALM_elastic_safeguard} are included in the total count 
since they are performed already by the inner solver to certificate that $x^k$ is an $\varepsilon_k$-approximate solution of \eqref{eq:subproblem}
in the sense of \eqref{eq:subproblem_stationarity}.

\section{Numerical experiments}\label{sec:num_experiments}

In this section, we analyze the practical performance of \cref{alg:ALM_elastic_safeguard}
based on the benchmark collection CUTEst \cite{Gould2014}.
First, we comment on some implementation details in \cref{sec:setup}.
Second, in \cref{sec:CUTEst}, the precise setup is presented,
and associated results are documented and analyzed.

\subsection{Implementation details}\label{sec:setup}

We implemented \cref{alg:ALM_elastic_safeguard} as a variant of ALGENCAN \cite{andreani2008augmented,birgin2014practical,birgin2020complexity},
which is a special instance of \cref{alg:ALM_rigid_safeguard}
with the safeguards from \eqref{eq:natural_choice_safeYY}. 
The only difference between the two methods lies in the way they handle the safeguards $\safeYY^g$ and $\safeYY^h$ defined in \eqref{eq:natural_choice_safeYY}.
Since $\mu_{\max}$, $\lambda_{\min}$, and $\lambda_{\max}$ are fixed vectors in ALGENCAN, we refer to it as Rigid ALGENCAN hereafter, or \ralgencan{} for short.
In practice, one typically sets $\lambda_{\min}=-\epsilon_{\mathrm{mach}}^{-1}\mathtt e$
and $\mu_{\max}=\lambda_{\max}=\epsilon_{\mathrm{mach}}^{-1}\mathtt e$, 
where $\epsilon_{\mathrm{mach}} \approx 10^{-16}$ denotes machine precision and $\mathtt e$ is the all-ones vector of suitable dimension. 
In contrast, the safeguards are dynamic and their update depends on the parameter $\eta$ in \cref{alg:ALM_elastic_safeguard}, which we refer to as Elastic ALGENCAN, or \ealgencan{} for short.
By default, ALGENCAN uses $\gamma=0.1$ when updating the penalty parameter. 
The relation $\eta \in (\gamma,\sqrt{\gamma})$ suggests that 0.2 and 0.3 are reasonable choices for $\eta$ in \cref{alg:ALM_elastic_safeguard}.
In the experiments, we also tested different choices for the parameter vectors $\mu_{\max}$, $\lambda_{\min}$, and $\lambda_{\max}$. 
Specifically, we always considered $\lambda_{\min}=-\zeta\mathtt e$ and $\mu_{\max}=\lambda_{\max}=\zeta\mathtt e$, and tested $\zeta \in \{1, 10^3, 10^6, 10^{12}\}$. 

ALGENCAN is an ALM that incorporates several additional mechanisms to enhance its efficiency and robustness.
Two such features are the feasibility restoration (attempting to find a feasible point by minimizing the constraint violation) and the use of acceleration strategies (attempting to solve the nonlinear KKT system using Newton's method).
In the experiments reported in this section, our objective was to analyze a specific aspect of the ALM itself. 
For this reason, both of the aforementioned strategies, which are enabled by default in ALGENCAN, were disabled.

All tests reported below were conducted on a computer with a 5.1 GHz Intel Core i9-12900K processor and 128GB 3200 MHz DDR4 RAM memory, running Ubuntu 22.04.3 LTS. 
Codes were written in Fortran and compiled by the GNU Fortran compiler of GCC (version 11.4.0) with the -O3 optimization directive enabled.

\subsection{Results on CUTEst benchmark collection}\label{sec:CUTEst}

In our experiments, we considered all problems from the CUTEst collection~\cite{Gould2014} featuring inequality or equality constraints, or both.
This resulted in a total of 774 test problems. 
For each problem, we used the default dimensions and the provided starting point.
The collected instances have $n \in \{1,\ldots, 250997\}$ variables, $m \in \{0,\ldots, 570781\}$ inequality constraints, and $\ell \in \{0,\ldots, 250498\}$ equality constraints.
The distribution of problem sizes is shown in \cref{fig:cutest_distribution}.

\begin{figure}[tbh]
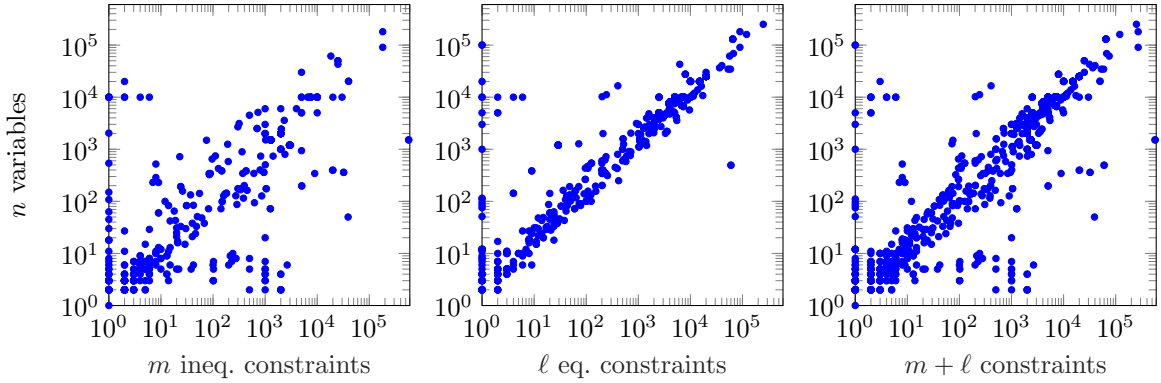

	\centering%
	\includetikz{cutest_distribution}%
	\caption{Distribution of test instances' size: number of variables $n$ against number of inequality constraints $m$ (left), number of equality constraints $\ell$ (middle), and total number of constraints $m+\ell$ (right).}%
	\label{fig:cutest_distribution}%
\end{figure}

Our comparisons focused primarily on robustness, namely the ability of the competing variants to identify feasible points with the lowest possible objective function values. 
As a secondary criterion, we assessed their ability to compute points satisfying the KKT conditions. 
For those problems in which the methods under comparison produced feasible points with similar objective function values, we compared their efficiency.
Since the methods being compared were merely variants of the same algorithm, we adopted the total number of objective function evaluations plus constraint evaluations as the performance metric. 
Because all variants performed exactly the same linear algebra computations, this metric provided a fair and reliable measure of computational effort, avoiding the measurement inaccuracies often associated with CPU-time comparisons.
To evaluate and compare efficiency of various methods, we use (relative) performance profiles \cite{Dolan2002}.
Exercising the caution warranted by Gould and Scott \cite{gould2016note},
we adopt also absolute performance profiles.
For a given method $i \in \{ 1,\dots,n_{\textsc{meth}}\}$,
the \emph{absolute} and \emph{relative performance profiles} are given by
\[
\Gamma_i^{\text{a}}(\tau) 
\coloneqq 
\frac{|\left\{ j \in \{1,\dots,n_{\textsc{prob}}\} \; | \; t_{ij} \leq \tau \right\}|}{n_{\textsc{prob}}}
\quad\text{and}\quad
\Gamma_i(\kappa) 
\coloneqq 
\frac{|\left\{ j \in \{1,\dots,n_{\textsc{prob}}\} \; | \; t_{ij}
	\leq \kappa t_j^\star \right\}|}{n_{\textsc{prob}}},
\]
respectively, where
$n_{\textsc{prob}}$ is the number of considered problems,
$|{\cal S}|$ denotes the cardinality of the set~${\cal S}$,
$t_{ij}$ is the performance metric for method~$i$ when applied to problem~$j$, 
and $t_j^\star \coloneqq \min_{s \in \{1,\dots,n_{\textsc{meth}}\}} \{ t_{sj} \}$ is the best performance among the $n_{\textsc{meth}}$ methods on problem $j$.

ALGENCAN declares success when it finds a primal iterate together with associated Lagrange multipliers 
that satisfy feasibility, complementarity, and optimality criteria within tolerances $\epsilon_{\mathrm{feas}}$, $\epsilon_{\mathrm{compl}}$, and $\epsilon_{\mathrm{opt}}$, respectively. 
The precise stopping criterion is described in \cite{birgin2020complexity}. 
In practice, we set all three tolerances equal to $10^{-8}$ and a CPU time limit of 10 minutes on each run. 
From this point onward we refer to a stationary point as any point that satisfies this stopping criterion, and to a feasible point as any point that satisfies at least the feasibility criterion. 
Suppose that, for a given problem, by running different methods, we obtain $n_{\textsc{feas}}$ feasible points $x_1, \dots, x_{n_{\textsc{feas}}}$ with objective function values $f(x_1), \dots, f(x_{n_{\textsc{feas}}})$. 
We say that $x_p$, $p\in\{1,\dots,n_{\textsc{feas}}\}$, is a solution if
\[
	f(x_p) \leq f_{\mathrm{best}} + f_{\mathrm{tol}} \max\{1, |f_{\mathrm{best}}|\},
\]
where $f_{\mathrm{best}} \coloneqq \min\{ f(x_1), \dots, f(x_{n_{\textsc{feas}}}) \}$ and $f_{\mathrm{tol}} > 0$ is a prescribed tolerance. 
When comparing different methods on a given problem, 
any method that fails to find a feasible point is considered unsuccessful whenever at least one competing method succeeds in finding one. 
Among the methods that do find feasible points, those that fail to find a solution, as defined above, are also considered unsuccessful. 
Efficiency is compared only among the methods that successfully find a solution.

\medskip

Considering Rigid ALGENCAN (\ralgencan{}) and the eight combinations of $\eta$ and $\zeta$ for Elastic ALGENCAN (\ealgencan{}), we applied nine variants of ALGENCAN to the 774 CUTEst problems. 
Our first observation is that the nine variants behaved identically on 424 of the 774 problems. 
Therefore, the comparison that follows focuses on the remaining 350 problems. 
On these problems, with the exception of the two variants using $\zeta=1$, the other seven variants found between 164 and 167 points satisfying the KKT conditions. 
The variants with $\zeta=1$ found 93 stationary points when $\eta=0.2$ and 131 stationary points when $\eta=0.3$. 
With respect to finding feasible points, all nine variants found between 296 and 302 feasible points. 
Up to this stage, no substantial differences can be observed among the methods. 
The comparison in terms of finding a solution depends on the value chosen for $f_{\mathrm{tol}}$. 
\cref{tab1} reports the results for values of $f_{\mathrm{tol}}$ ranging from $10^{-1}$ to $10^{-8}$. 
For relatively loose tolerances, no significant differences can be observed among the methods. 
However, for $f_{\mathrm{tol}}=10^{-8}$, for example, the difference between the variants using $\zeta \in \{1,10^3\}$ and those using $\zeta \in \{10^6,10^{12}\}$ is approximately ten problems. 
In other words, the variants with smaller values of $\zeta$ find roughly ten fewer solutions than the other variants.
\ralgencan{} is consistently among the variants that find the largest number of solutions, regardless of the value of $f_{\mathrm{tol}}$. 
Overall, considering the criteria of finding stationary points, finding feasible points, and finding solutions, the \ealgencan{} variants with $\eta=0.3$ exhibit a marginal advantage over those using $\eta=0.2$. 
Among the \ealgencan{} variants with $\eta=0.3$, those using $\zeta \in \{10^6,10^{12}\}$ appear to exhibit a minor advantage over \ralgencan{}.

\begin{table}[ht!]
\renewcommand*{\arraystretch}{1.1}
\begin{center}
\begin{tabular}{ccc|cccccccc}
\hline
& \multirow{2}{*}{$\eta$} & \multirow{2}{*}{$\zeta$} & \multicolumn{8}{c}{$f_{\mathrm{tol}}$} \\
\cline{4-11}
&  &  & $10^{-1}$ & $10^{-2}$ & $10^{-3}$ & $10^{-4}$ & $10^{-5}$ & $10^{-6}$ & $10^{-7}$ & $10^{-8}$ \\
\hline
\hline
\multirow{8}{*}{\rotatebox{90}{Elastic ALGENCAN}}
& \multirow{4}{*}{$0.2$} & $1$       & 289 		& 287 		& \bf{281} 	& 274 		& 269 		& \bf{265} 	& \bf{255} 	& 233\\
&                        & $10^3$    & 288 		& 284 		& 280 		& 275 		& 270 		& 262 		& 246 		& 234\\
&                        & $10^6$    & 291 		& 286 		& 279 		& 277 		& 272 		& 263 		& 250 		& 240\\
&                        & $10^{12}$ & 290 		& 286 		& 279 		& 277 		& 273 		& 264 		& 251 		& 241\\
\cline{2-11}
& \multirow{4}{*}{$0.3$} & $1$       & 291 		& \bf{290} 	& \bf{281} 	& 273 		& 267 		& 261 		& 252 		& 234\\
&                        & $10^3$    & 290 		& 286 		& 280 		& 275 		& 270 		& 261 		& 248 		& 235\\
&                        & $10^6$    & \bf{292} & 288 		& \bf{281} 	& \bf{279} 	& \bf{274} 	& \bf{265} 	& 252 		& \bf{242} \\
&                        & $10^{12}$ & 291 		& 287 		& 280 		& 278 		& \bf{274} 	& \bf{265} 	& 252 		& \bf{242} \\
\hline
\hline
\multicolumn{3}{l|}{Rigid ALGENCAN}    & 290 	& 286 		& 279 		& 277 		& 273 		& 264 		& 251 		& 241\\
\hline
\end{tabular}
\end{center}
\caption{Number of solutions found by Rigid ALGENCAN (\ralgencan{}) and Elastic ALGENCAN (\ealgencan{}) for different values of $\eta$ and $\zeta$, for varying values of $f_{\mathrm{tol}}$. For each column, the best performance is highlighted in \textbf{boldface}.}
\label{tab1}
\end{table}

If, motivated by their slightly superior robustness, we restrict the comparison to \ralgencan{} and the \ealgencan{} variants with $\eta=0.3$ and $\zeta \in \{10^6,10^{12}\}$, 
the three methods behave identically on 745 of the 774 problems, leaving only 29 problems on which their performance differs. 
On these problems, they find essentially the same number of stationary points (either none or one), the same number of feasible points (between 11 and 13), 
and the same number of solutions (between 8 and 10). 
On 7 problems, all three methods find equivalent solutions. 
For each of these 7 problems, \ealgencan{} with $\eta=0.3$ and $\zeta=10^6$ is the most efficient variant.
On the other hand, for any value of $f_{\mathrm{tol}}$ strictly greater than $10^{-8}$, \cref{tab1} shows that there are no significant differences in the number of solutions found by the nine variants. 
Moreover, among the 350 problems on which the variants do not behave identically, there are 286 for which all nine variants find a solution. 
\cref{fig1} compares the efficiency of the nine variants on these 286 problems using profiles 
that consider, as performance measure, the number of objective function evaluations plus constraint evaluations performed by each method.
The figure shows that all variants exhibit very similar efficiency, with a slight disadvantage for \ealgencan{} with $\eta=0.2$ and $\zeta=1$ and, 
perhaps surprisingly, a slight advantage for \ealgencan{} with $\eta=0.3$ and $\zeta=1$.

\begin{figure}[tbh]
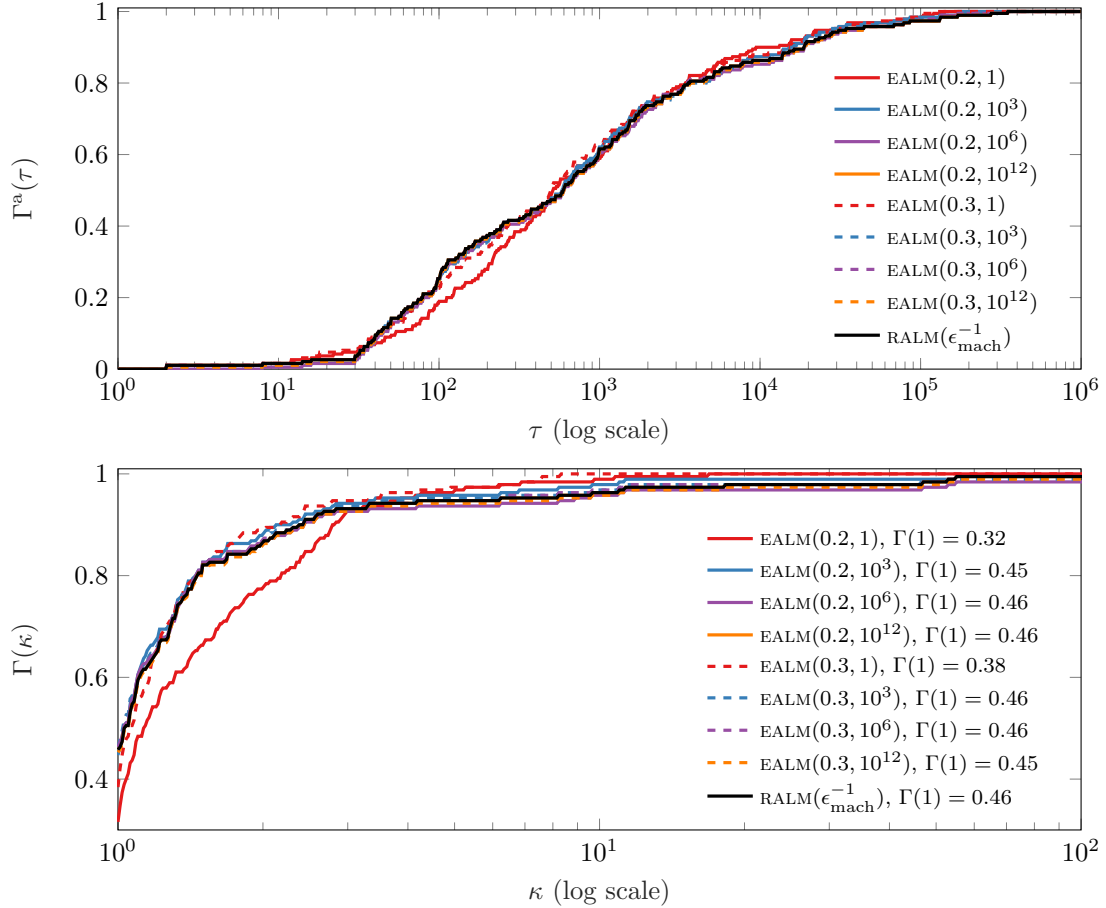

	\centering%
	\includetikz{dataprof_1}
	\includetikz{perfprof_1}
	\caption{Comparison of the eight variants of Elastic ALGENCAN (denoted \ealgencan{}$(\eta,\zeta)$) and the default version of Rigid ALGENCAN (denoted \ralgencan{}$(\zeta)$), considering the 286 problems in which all nine methods found similar solutions. Absolute (top) and relative (bottom) performance profiles.}
	\label{fig1}
\end{figure}

At this point, it is clear that the performance of \ealgencan{} is very similar to that of \ralgencan{} and that \ealgencan{} is largely insensitive to the choice of $\eta$ and $\zeta$. 
A closer assessment of \ealgencan{} with $\eta=0.3$ and $\zeta=1$ relative to \ralgencan{} is depicted in \cref{fig3}, considering the 173 problems for which the two methods found comparable solutions.
These profiles indicate that there is a small yet perceptible benefit with \ealgencan{} over \ralgencan{}.

\begin{figure}[tbh]
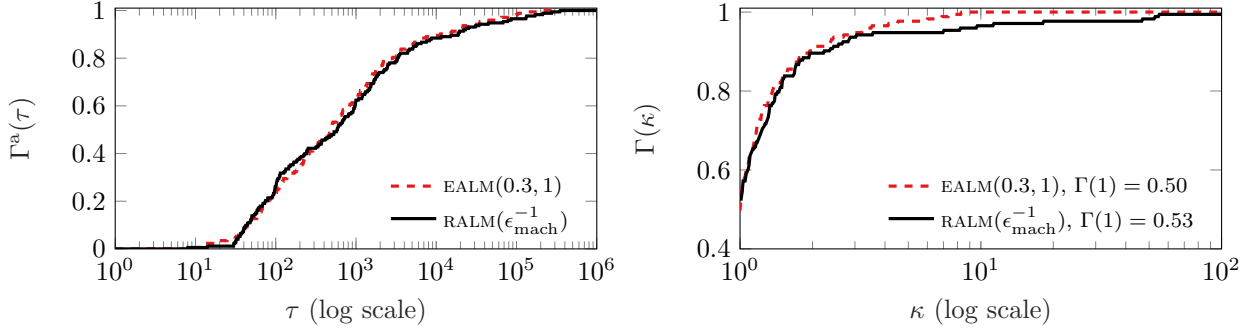

	\centering%
	\includetikz{dataprof_3}%
	\includetikz{perfprof_3}
	\caption{Comparison of Elastic ALGENCAN with $\eta=0.3$ and $\zeta=1$ (denoted \ealgencan{}$(\eta,\zeta)$) and the default Rigid ALGENCAN (denoted \ralgencan{}$(\zeta)$), considering the 173 problems in which the two methods found similar solutions. Absolute (left) and relative (right) performance profiles.}
	\label{fig3}
\end{figure}

By contrast, we now consider \ralgencan{} variants with $\zeta \in \{0, 1, 10^3, 10^6, 10^{12}, \epsilon_{\mathrm{mach}}^{-1}\}$. 
The first corresponds to the pure quadratic penalty method, while the last corresponds to the default version of \ralgencan{}. 
These six variants behave identically on only 50 of the 774 problems. 
Considering the remaining 724 problems, the variants with $\zeta=0$, $\zeta=1$, and $\zeta=10^3$ find 136, 309, and 457 stationary points, respectively, whereas each of the remaining variants finds 483 stationary points. 
The number of feasible points found across the six variants ranges from 662 to 670. 
Similarly, when $f_{\mathrm{tol}}=0.1$, the number of solutions found ranges from 648 to 654. 
For the 645 problems on which all six variants find a solution, their efficiency can be meaningfully compared.
\cref{fig2} presents the corresponding performance profiles, revealing that \ralgencan{} is more sensitive to the choice of safeguard bounds than \ealgencan{}.

\begin{figure}[tbh]
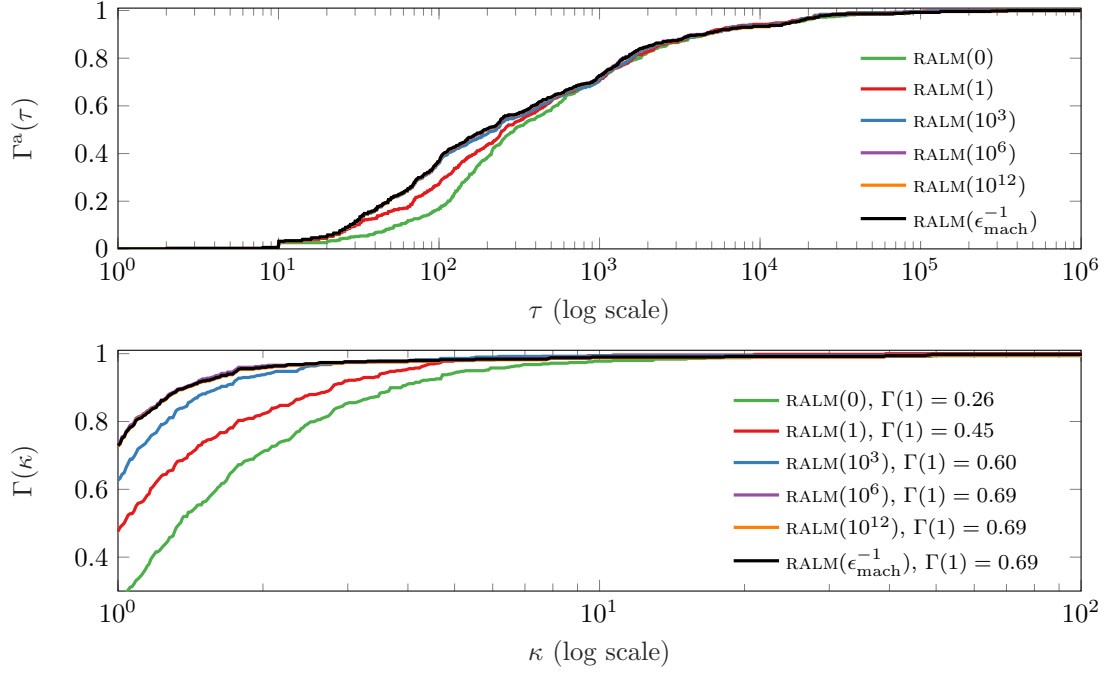

	\centering
	\includetikz{dataprof_2}
	\includetikz{perfprof_2}
	\caption{Comparison of six variants of Rigid ALGENCAN (denoted \ralgencan{}$(\zeta)$), including pure quadratic penalty ($\zeta=0$), considering the 645 problems in which all six methods found similar solutions. Absolute (top) and relative (bottom) performance profiles.}
	\label{fig2}
\end{figure}

\section{Conclusions}\label{sec:conclusions}

In this work, we investigated elastic safeguarding in augmented Lagrangian methods (ALM) and studied it through the lens of ALGENCAN, a state-of-the-art solver against which its elastically safeguarded variant was compared.
From a practical standpoint the main finding is that the elastic ALM achieves robustness and efficiency nearly indistinguishable from those of its rigidly safeguarded counterpart, while being largely insensitive to the choice of its parameters.
The latter property is not shared by the rigid variant, whose performance deteriorates when small safeguarding bounds are used.
The elastic ALM is therefore the preferable alternative.
As illustrated in this paper, from a theoretical standpoint both methods share essentially the same global convergence and complexity guarantees
when considering standard nonlinear optimization problems, confirming that the added flexibility of elastic safeguarding incurs no theoretical cost.
In \cite[\S~6]{demarchi2025augmented}, it has been illustrated that, when restricting to fully convex optimization problems,
the elastic ALM possesses some advantages over its rigid analogue.
Thus, the elastic mechanism is recommended when implementing safeguarded ALMs.

% references
\phantomsection
\addcontentsline{toc}{section}{References}%
\bibliographystyle{habbrv}
\bibliography{bdm2026}

\begin{thebibliography}{10}
\expandafter\ifx\csname url\endcsname\relax
  \def\url#1{\texttt{#1}}\fi
\expandafter\ifx\csname doi\endcsname\relax
  \def\doi#1{\burlalt{doi:#1}{http://dx.doi.org/#1}}\fi
\expandafter\ifx\csname urlprefix\endcsname\relax\def\urlprefix{URL }\fi
\expandafter\ifx\csname href\endcsname\relax
  \def\href#1#2{#2}\fi
\expandafter\ifx\csname burlalt\endcsname\relax
  \def\burlalt#1#2{\href{#2}{#1}}\fi

\bibitem{andreani2008augmented}
R.~Andreani, E.~G. Birgin, J.~M. Mart{\'i}nez, and M.~L. Schuverdt.
\newblock On augmented {L}agrangian methods with general lower--level
  constraints.
\newblock {\em SIAM Journal on Optimization}, 18(4):1286--1309, 2008.
\newblock \doi{10.1137/060654797}.

\bibitem{AndreaniMartinezRamosSilva2016}
R.~Andreani, J.~M. Mart{\'{i}}nez, A.~Ramos, and P.~J.~S. Silva.
\newblock A cone-continuity constraint qualification and algorithmic
  consequences.
\newblock {\em SIAM Journal on Optimization}, 26(1):96--110, 2016.
\newblock \doi{10.1137/15M1008488}.

\bibitem{bertsekas1996constrained}
D.~P. Bertsekas.
\newblock {\em Constrained Optimization and {L}agrange Multiplier Methods}.
\newblock Athena Scientific, 1996.

\bibitem{BirginCastilloMartinez2005}
E.~G. Birgin, R.~A. Castillo, and J.~M. Mart\'inez.
\newblock Numerical comparison of augmented {L}agrangian algorithms for
  nonconvex problems.
\newblock {\em Computational Optimization and Applications}, 31:31--55, 2005.
\newblock \doi{10.1007/s10589-005-1066-7}.

\bibitem{BirginHaeserMaculanRamirez2025}
E.~G. Birgin, G.~Haeser, N.~Maculan, and L.~M. Ramirez.
\newblock On the global convergence of a general class of augmented
  {L}agrangian methods.
\newblock {\em Journal of Optimization Theory and Applications}, 206:57, 2025.
\newblock \doi{10.1007/s10957-025-02734-0}.

\bibitem{birgin2014practical}
E.~G. Birgin and J.~M. Mart\'inez.
\newblock {\em Practical Augmented {L}agrangian Methods for Constrained
  Optimization}.
\newblock SIAM, Philadelphia, 2014.
\newblock \doi{10.1137/1.9781611973365}.

\bibitem{birgin2020complexity}
E.~G. Birgin and J.~M. Mart\'inez.
\newblock Complexity and performance of an augmented {Lagrangian} algorithm.
\newblock {\em Optimization Methods and Software}, 35(5):885--920, 2020.
\newblock \doi{10.1080/10556788.2020.1746962}.

\bibitem{conn1991globally}
A.~R. Conn, N.~I.~M. Gould, and P.~L. Toint.
\newblock A globally convergent augmented {L}agrangian algorithm for
  optimization with general constraints and simple bounds.
\newblock {\em SIAM Journal on Numerical Analysis}, 28(2):545--572, 1991.
\newblock \doi{10.1137/0728030}.

\bibitem{demarchi2025augmented}
A.~De~Marchi, T.~Hoheisel, and P.~Mehlitz.
\newblock Augmented {L}agrangian methods for fully convex composite
  optimization, 2025, \burlalt{2511.07117}{http://arxiv.org/abs/2511.07117}.

\bibitem{demarchi2023constrained}
A.~De~Marchi, X.~Jia, C.~Kanzow, and P.~Mehlitz.
\newblock Constrained composite optimization and augmented {L}agrangian
  methods.
\newblock {\em Mathematical Programming}, 201(1):863--896, 2023.
\newblock \doi{10.1007/s10107-022-01922-4}.

\bibitem{demarchi2024local}
A.~De~Marchi and P.~Mehlitz.
\newblock Local properties and augmented {L}agrangians in fully nonconvex
  composite optimization.
\newblock {\em Journal of Nonsmooth Analysis and Optimization}, 5, 2024.
\newblock \doi{10.46298/jnsao-2024-12235}.

\bibitem{Dolan2002}
E.~D. Dolan and J.~J. Mor\'e.
\newblock Benchmarking optimization software with performance profiles.
\newblock {\em Mathematical Programming}, 91(2):201--213, 2002.
\newblock \doi{10.1007/s101070100263}.

\bibitem{Gould2014}
N.~I.~M. Gould, D.~P. Orban, and P.~L. Toint.
\newblock {CUTEst}: a constrained and unconstrained testing environment with
  safe threads for mathematical optimization.
\newblock {\em Computational Optimization and Applications}, 60(3):545--557,
  2014.
\newblock \doi{10.1007/s10589-014-9687-3}.

\bibitem{gould2016note}
N.~I.~M. Gould and J.~Scott.
\newblock A note on performance profiles for benchmarking software.
\newblock {\em ACM Transactions on Mathematical Software}, 43(2), 2016.
\newblock \doi{10.1145/2950048}.

\bibitem{hestenes1969multiplier}
M.~R. Hestenes.
\newblock Multiplier and gradient methods.
\newblock {\em Journal of Optimization Theory and Applications}, 4(5):303--320,
  11 1969.
\newblock \doi{10.1007/BF00927673}.

\bibitem{KanzowKraemerMehlitzWachsmuthWerner2025}
C.~Kanzow, F.~Kr{\"a}mer, P.~Mehlitz, G.~Wachsmuth, and F.~Werner.
\newblock Variational {P}oisson denoising via augmented {L}agrangian methods.
\newblock {\em Electronic Transactions on Numerical Analysis}, 63:33--62, 2025.
\newblock \doi{10.1553/etna{\_}vol63s33}.

\bibitem{kanzow2017example}
C.~Kanzow and D.~Steck.
\newblock An example comparing the standard and safeguarded augmented
  {L}agrangian methods.
\newblock {\em Operations Research Letters}, 45(6):598--603, 2017.
\newblock \doi{10.1016/j.orl.2017.09.005}.

\bibitem{powell1969method}
M.~J.~D. Powell.
\newblock {\em A method for nonlinear constraints in minimization problems},
  pages 283--298.
\newblock Academic Press, 1969.

\bibitem{rockafellar1973multiplier}
R.~T. Rockafellar.
\newblock The multiplier method of {H}estenes and {P}owell applied to convex
  programming.
\newblock {\em Journal of Optimization Theory and Applications},
  12(6):555--562, 1973.
\newblock \doi{10.1007/BF00934777}.

\bibitem{rockafellar1976augmented}
R.~T. Rockafellar.
\newblock Augmented {L}agrangians and applications of the proximal point
  algorithm in convex programming.
\newblock {\em Mathematics of operations research}, 1(2):97--116, 1976.
\newblock \doi{10.1287/moor.1.2.97}.

\bibitem{rockafellar1976monotone}
R.~T. Rockafellar.
\newblock Monotone operators and the proximal point algorithm.
\newblock {\em SIAM Journal on Control and Optimization}, 14(5):877--898, 1976.
\newblock \doi{10.1137/0314056}.

\bibitem{rockafellar2022convergence}
R.~T. Rockafellar.
\newblock Convergence of augmented {L}agrangian methods in extensions beyond
  nonlinear programming.
\newblock {\em Mathematical Programming}, 199(1):375--420, 2023.
\newblock \doi{10.1007/s10107-022-01832-5}.

\bibitem{rockafellar1998variational}
R.~T. Rockafellar and R.~J. Wets.
\newblock {\em Variational Analysis}, volume 317.
\newblock Springer, 1998.
\newblock \doi{10.1007/978-3-642-02431-3}.

\end{thebibliography}

\end{document}